\documentclass{article}

\usepackage[utf8]{inputenc} 
\usepackage[T1]{fontenc}    
\usepackage{url}            
\usepackage{booktabs}       
\usepackage{amsfonts}       
\usepackage{nicefrac}       
\usepackage{microtype}      
\usepackage{setspace}
\usepackage[round]{natbib}
\usepackage[colorlinks=true,citecolor=blue,urlcolor=blue,linkcolor=blue]{hyperref}
\usepackage[left=1.25in,top=1.25in,right=1.25in,bottom=1.25in,head=1.25in]{geometry}

\usepackage{amssymb,amsmath,amsthm}
\usepackage{enumerate,dsfont}
\usepackage{graphicx}
\newtheorem{theorem}{Theorem}
\newtheorem{lemma}{Lemma}
\newtheorem{corollary}{Corollary}

\theoremstyle{definition}
\newtheorem{remark}{Remark}
\newtheorem{definition}{Definition}

\newcommand{\argmin}{\mathop{\mathrm{arg\,min}}}

\newcommand{\minimize}{\mathop{\mathrm{minimize}}}
\newcommand{\maximize}{\mathop{\mathrm{maximize}}}
\newcommand{\st}{\mathop{\mathrm{subject\,\,to}}}

\def\R{\mathbb{R}}

\def\P{\mathbb{P}}

\def\dim{\mathrm{dim}}
\def\col{\mathrm{col}}
\def\row{\mathrm{row}}
\def\nul{\mathrm{null}}
\def\rank{\mathrm{rank}}
\def\nuli{\mathrm{nullity}}
\def\spa{\mathrm{span}}
\def\sign{\mathrm{sign}}
\def\supp{\mathrm{supp}}
\def\diag{\mathrm{diag}}
\def\aff{\mathrm{aff}}

\def\ones{\mathds{1}}
\def\hbeta{\hat{\beta}}
\def\hgamma{\hat{\gamma}}
\def\hu{\hat{u}}
\def\hv{\hat{v}}
\def\cA{\mathcal{A}}
\def\cB{\mathcal{B}}
\def\cN{\mathcal{N}}
\def\ie{i.e.}
\def\eg{e.g.}

\def\inte{\mathrm{int}}
\def\bd{\mathrm{bd}}

\def\relbd{\mathrm{relbd}}

\def\dom{\mathrm{dom}}
\def\ran{\mathrm{ran}}
\def\S{\mathbb{S}}
\def\B{\mathbb{B}}
\def\cH{\mathcal{H}}
\def\cL{\mathcal{L}}
\def\cone{\mathrm{cone}}

\title{The Generalized Lasso Problem and Uniqueness}
\author{Alnur Ali$^1$ \and Ryan J. Tibshirani$^{1,2}$}
\date{$^1$Machine Learning Department, Carnegie Mellon University \\
$^2$Department of Statistics, Carnegie Mellon University}

\begin{document}
\maketitle

\begin{abstract}
We study uniqueness in the generalized lasso problem, where the penalty is the
$\ell_1$ norm of a matrix $D$ times the coefficient vector.  We derive a broad
result on uniqueness that places weak assumptions on the predictor matrix $X$
and penalty matrix $D$; the implication is that, if $D$ is fixed and its null
space is not too large (the dimension of its null space is at most the
number of samples), and $X$ and response vector $y$ jointly follow an
absolutely continuous distribution, then the generalized lasso problem has a
unique solution almost surely, regardless of the number of predictors relative
to the number of samples.  This effectively generalizes previous
uniqueness results for the lasso problem \citep{tibshirani2013lasso}
(which corresponds to the special case $D=I$).  Further, we extend our study to 
the case in which the loss is given by the negative log-likelihood from a
generalized linear model. In addition to uniqueness results, we derive 
results on the local stability of generalized lasso solutions that might be of
interest in their own right.  
\end{abstract}

\section{Introduction}
\label{sec:introduction}

We consider the {\it generalized lasso} problem 
\begin{equation}
\label{eq:genlasso}
\minimize_{\beta \in \R^p} \; \frac{1}{2} \| y - X \beta \|_2^2 + 
\lambda \| D \beta \|_1,
\end{equation}
where $y \in \R^n$ is a response vector, $X \in \R^{n\times p}$ is a predictor
matrix, $D \in \R^{m\times p}$ is a penalty matrix, and $\lambda \geq 0$ is a 
tuning parameter.  As explained in \citet{tibshirani2011solution}, the
generalized lasso problem \eqref{eq:genlasso} encompasses several well-studied 
problems as special cases, corresponding to different choices of $D$, \eg,  
the lasso \citep{tibshirani1996regression}, 
the fused lasso \citep{rudin1992nonlinear,tibshirani2005sparsity}, 
trend filtering \citep{steidl2006splines,kim2009ell1},
the graph fused lasso \citep{hoefling2010path},
graph trend filtering \citep{wang2016trend}, 
Kronecker trend filtering \citep{sadhanala2017higher},
among others. (For all problems except the lasso problem, the literature is 
mainly focused on the so-called ``signal approximator'' case, where
$X=I$, and the responses have a certain underlying structure; but the
``regression'' case, where $X$ is arbitrary, naturally arises whenever the
predictor variables---rather than the responses---have an analogous structure.)        

There has been an abundance of theoretical and computational work on the
generalized lasso and its special cases.  In the current paper, we examine
sufficient conditions under which the solution in 
\eqref{eq:genlasso} will be unique. While this is simple enough to state, it is a
problem of fundamental importance.  The generalized lasso has been used as a
modeling tool in numerous application areas, such as copy number variation 
analysis \citep{tibshirani2008spatial}, sMRI image classification
\citep{xin2014efficient}, evolutionary shift detection on phylogenetic trees, 
\citep{khabbazian2016fast}, motion-capture tracking \citep{padilla2017tensor},
and longitudinal prediction of disease progression \citep{adhikari2019high}.  
In such applications, the structure of the solution \smash{$\hbeta$}
in hand (found by using one of many optimization methods 
applicable to \eqref{eq:genlasso}, a convex quadratic program) usually
carries meaning---this is because $D$ has been  
carefully chosen so that sparsity in \smash{$D\hbeta$} translates into some
interesting and domain-appropriate structure for \smash{$\hbeta$}.  Of course, 
nonuniqueness of the solution in \eqref{eq:genlasso} would cause complications 
in interpreting this structure. (The practitioner would be left wondering: are
there other solutions providing compementary, or even contradictory structures?) 
Further, beyond interpretation, nonuniqueness of the generalized lasso solution
would clearly cause complications if we are seeking to use this solution to make
predictions (via \smash{$x^T \hbeta$}, for a new predictor vector $x \in \R^p$),
as different solutions would lead to different predictions (potentially very
different ones).  

When $p \leq n$ and $\rank(X)=p$, there is always a unique
solution in \eqref{eq:genlasso} due to strict convexity of the squared loss
term.  Our focus will thus be in deriving sufficient conditions for
uniqueness in the high-dimensional case, where $\rank(X)<p$.  
It also worth noting that when $\nul(X) \cap \nul(D) \neq \{0\}$ problem 
\eqref{eq:genlasso} cannot have a unique solution.  (If $\eta \neq 0$ lies in
this intersection, and \smash{$\hbeta$} is a solution in \eqref{eq:genlasso},
then so will be \smash{$\hbeta+\eta$}.) Therefore, at the very least, any
sufficient condition for uniqueness in \eqref{eq:genlasso} must include (or
imply) the null space condition $\nul(X) \cap \nul(D) = \{0\}$.  

In the lasso problem, defined by taking $D=I$ in \eqref{eq:genlasso}, several 
authors have studied conditions for uniqueness, notably 
\citet{tibshirani2013lasso}, who showed that when the entries of $X$ are
drawn from an arbitrary continuous distribution, the lasso solution is unique
almost surely.  One of the main results in this paper yields this lasso result
as a special case; see Theorem \ref{thm:uniqueness}, and Remark
\ref{rem:full_row_rank} following the theorem.  Moreover, our study of
uniqueness leads us to develop intermediate properties of generalized lasso
solutions that may be of interest in their own right---in particular, when we
broaden our focus to a version of \eqref{eq:genlasso} in which the squared loss
is replaced by a general loss function, we derive local stability properties of 
solutions that have potential applications beyond this paper.

In the remainder of this introduction, we describe the implications of our
uniqueness results for various special cases of the generalized lasso,
discuss related work, and then cover notation and an outline of the rest of the 
paper.  

\subsection{Uniqueness in special cases}

The following is an application of Theorem \ref{thm:uniqueness} to various
special cases for the penalty matrix $D$.  The takeaway is that, for
continuously distributed predictors and responses, uniqueness can be ensured
almost surely in various interesting cases of the generalized lasso, provided
that $n$ is not ``too small'', meaning that the sample size $n$ is at least the
nullity (dimension of the null space) of $D$.  (Some of the cases presented in
the corollary can be folded into others, but we list them anyway for clarity.)      

\begin{corollary}
\label{cor:uniqueness}
Fix any $\lambda>0$.  Assume the joint distribution of $(X,y)$ is absolutely
continuous with respect to $(np+n)$-dimensional Lebesgue measure.  Then 
problem \eqref{eq:genlasso} admits a unique solution almost surely, in any one
of the following cases: 
\begin{enumerate}[(i)]
\item $D=I \in \R^{p \times p}$ is the identity matrix; 
\item $D \in \R^{(p-1) \times p}$ is the first difference matrix, i.e., fused
  lasso penalty matrix (see Section 2.1.1 in \citet{tibshirani2011solution});  
\item $D \in \R^{(p-k-1) \times p}$ is the $(k+1)$st order difference matrix,
  i.e., $k$th order trend filtering penalty matrix (see Section 2.1.2 in
  \citet{tibshirani2011solution}), and $n \geq k+1$;   
\item $D \in \R^{m \times p}$ is the graph fused lasso penalty matrix, defined
  over a graph with $m$ edges, $n$ nodes, and $r$ connected components (see
  Section 2.1.1 in \citet{tibshirani2011solution}), and $n \geq r$;  
\item $D \in \R^{m \times p}$ is the $k$th order graph trend filtering penalty 
  matrix, defined over a graph with $m$ edges, $n$ nodes, and $r$ connected
  components (see \citet{wang2016trend}), and $n \geq r$;  
\item $D \in \R^{(N-k-1)N^{d-1}d \times N^d}$ is the $k$th order Kronecker trend
  filtering penalty matrix, defined over a $d$-dimensional grid graph with all
  equal side lengths $N=n^{1/d}$ (see \citet{sadhanala2017higher}), and $n 
  \geq (k+1)^d$. 
\end{enumerate}
\end{corollary}

Two interesting special cases of the generalized lasso that fall outside the
scope of our results here are {\it additive trend filtering}
\citep{sadhanala2017additive} and {\it varying-coefficient models} 
(which can be cast in a generalized lasso form, see Section 2.2 of
\citet{tibshirani2011solution}).  In either of these problems, the predictor
matrix $X$ has random elements but obeys a particular structure, thus it is not 
reasonable to assume that its entries overall follow a continuous distribution,
so Theorem \ref{thm:uniqueness} cannot be immediately applied.  Still, we
believe that under weak conditions either problem should
have a unique solution.  \citet{sadhanala2017additive} give a uniqueness 
result for additive trend filtering by reducing this problem to lasso form; but, 
keeping this problem in generalized lasso form and carefully
investigating an application of Lemma \ref{lem:d_gp} (the deterministic result 
in this paper leading to Theorem \ref{thm:uniqueness}) may yield a result
with simpler sufficient conditions.  This is left to future work. 

Furthermore, by applying Theorem \ref{thm:uniqueness_glm} to various special
cases for $D$, analogous results hold (for all cases in Corollary
\ref{cor:uniqueness}) when the squared loss is replaced by a generalized linear 
model (GLM) loss $G$ as in \eqref{eq:g_glm}.  In this setting, the assumption
that $(X,y)$ is jointly absolutely continuous is replaced by the two assumptions 
that $X$ is absolutely continous, and $y \notin \cN$, where $\cN$ is the set
defined in \eqref{eq:n}.  The set $\cN$ has Lebesgue measure zero for some
common choices of loss $G$ (see Remark \ref{rem:n}); but unless we somewhat
artificially assume that the distribution of $y|X$ is continuous (this is
artificial because in the two most fundamental GLMs outside of the Gaussian
model, namely the Bernoulli and Poisson models, the entries of $y|X$ are
discrete), the fact that $\cN$ has Lebesgue measure zero set does not directly
imply that the condition $y \notin \cN$ holds almost surely.  Still, it seems
that $y \notin \cN$ should be ``likely''---and hence, uniqueness should be
``likely''---in a typical GLM setup, and making this precise is left to future
work.    

\subsection{Related work}

Several authors have examined uniqueness of solutions in statistical
optimization problems en route to proving risk or recovery properties of these 
solutions; see \citet{donoho2006most,dossal2012necessary} for examples of this
in the noiseless lasso problem (and the analogous noiseless $\ell_0$ penalized
problem); see \citet{nam2013cosparse} for an example in the noiseless
generalized lasso problem; see
\citet{fuchs2005recovery,candes2009near,wainwright2009sharp} for examples   
in the lasso problem; and lastly, see \citet{lee2013model} for an example in the 
generalized lasso problem. These results have a different aim than ours,
i.e., their main goal---a risk or recovery guarantee---is more ambitious 
than certifying uniqueness alone, and thus the conditions they require are  
more stringent.  Our work in this paper is more along the lines of direct
uniqueness analysis in the lasso, as was carried out by \citet{osborne2000lasso, 
rosset2004boosting,tibshirani2013lasso,schneider2017distribution}.

\subsection{Notation and outline}

In terms of notation, for a matrix $A \in \R^{m \times n}$, we write $A^+$ for
its Moore-Penrose pseudoinverse and $\col(A),\row(A),\nul(A),\rank(A)$ for
its column space, row space, null space, and rank, respectively.  We write 
$A_J$ for the submatrix defined by the rows of $A$ indexed by a subset $J 
\subseteq \{1,\ldots,m\}$, and use \smash{$A_{-J}$} as shorthand for
\smash{$A_{\{1,\ldots,m\} \setminus J}$}. Similarly, for a vector $x \in \R^m$,
we write $x_J$ for the subvector defined by the components of $x$ 
indexed by $J$, and use \smash{$x_{-J}$} as shorthand for
\smash{$x_{\{1,\ldots,m\} \setminus J}$}. 

For a set $S \subseteq \R^n$, we write $\spa(S)$ for its linear span, and
write $\aff(S)$ for its affine span.  For a subspace $L \subseteq \R^n$, we  
write $P_L$ for the (Euclidean) projection operator onto $L$, and write 
\smash{$P_{L^\perp}$} for the projection operator onto the
orthogonal complement $L^\perp$.  For a function $f : \R^m \to \R^n$, we 
write $\dom(f)$ for its domain, and $\ran(f)$ for its range.

Here is an outline for what follows. In Section
\ref{sec:preliminaries}, we review important preliminary facts about the
generalized lasso.  In Section \ref{sec:uniqueness}, we derive sufficient 
conditions for uniqueness in \eqref{eq:genlasso}, culminating in Theorem
\ref{thm:uniqueness}, our main result on uniqueness in the squared 
loss case.  In Section \ref{sec:general_loss}, we consider a generalization of 
problem \eqref{eq:genlasso} where the squared loss is replaced by a 
smooth and strictly convex function of $X\beta$; we derive analogs of the
important preliminary facts used in the squared loss case, notably, we
generalize a result on the local stability of generalized lasso solutions due to
\citet{tibshirani2012degrees}; and we give  
sufficient conditions for uniqueness, culminating in Theorem
\ref{thm:uniqueness_glm}, our main result in the general loss case.   In Section
\ref{sec:discussion}, we conclude with a brief discussion.

\section{Preliminaries}
\label{sec:preliminaries}

\subsection{Basic facts, KKT conditions, and the dual} 

First, we establish some basic properties of the generalized lasso problem 
\eqref{eq:genlasso} relating to uniqueness.   

\begin{lemma}
\label{lem:basic}
For any $y,X,D$, and $\lambda \geq 0$, the following holds of the
generalized lasso problem \eqref{eq:genlasso}.
\begin{enumerate}[(i)]
\item There is either a unique solution, or uncountably many solutions.
\item Every solution \smash{$\hbeta$} gives rise to the same fitted value 
  \smash{$X \hbeta$}. 
\item If $\lambda>0$, then every solution \smash{$\hbeta$} gives rise to the 
  same penalty value \smash{$\|D\hbeta\|_1$}. 
\end{enumerate}
\end{lemma}

\begin{proof}
The criterion function in the generalized lasso problem \eqref{eq:genlasso} is
convex and proper, as well as closed (being continuous on $\R^p$).
As both $g(\beta)=\|y-X\beta\|_2^2$ and $h(\beta)=\lambda\|D\beta\|_1$ are 
nonnegative, any directions of recession of the criterion $f=g+h$ are
necessarily directions of recession of both $g$ and $h$.  Hence, we see that all 
directions of recession of the criterion $f$ must lie in the common null space 
$\nul(X) \cap \nul(D)$; but these are directions in which the criterion is
constant.  Applying, \eg, Theorem 27.1 in \citet{rockafellar1970convex} tells
us that the criterion attains its infimum, so there is at least one solution in
problem \eqref{eq:genlasso}. Supposing there are two solutions
\smash{$\hbeta^{(1)},\hbeta^{(2)}$}, since the solution set to a convex
optimization problem is itself a convex set, we get that \smash{$t \hbeta^{(1)}
  + (1-t) \hbeta^{(2)}$} is also a solution, for any $t \in [0,1]$.  Thus
if there is more than one solution, then there are uncountably many solutions.
This proves part (i). 

As for part (ii), let \smash{$\hbeta^{(1)},\hbeta^{(2)}$} be two solutions in
\eqref{eq:genlasso}, with \smash{$\hbeta^{(1)} \neq \hbeta^{(2)}$}.  Let
$f^\star$ denote the optimal criterion value in \eqref{eq:genlasso}.  Proceeding
by contradiction, suppose that these two solutions do not yield the same fit,
\ie, \smash{$X \hbeta^{(1)} \neq X \hbeta^{(2)}$}.  Then for any $t \in (0,1)$,
the criterion at \smash{$t \hbeta^{(1)} + (1-t) \hbeta^{(2)}$} is
\begin{align*}
f\big(t \hbeta^{(1)} + (1-t) \hbeta^{(2)}\big) 
&= \frac{1}{2} \big\| y - \big( t X \hbeta^{(1)} + (1-t) X
\hbeta^{(2)} \big) \big\|_2^2 + 
\lambda \big\| D \big( t \hbeta^{(1)} + (1-t) \hbeta^{(2)} \big) \big\|_1 \\    
&< t \frac{1}{2} \| y - X \hbeta^{(1)} \|_2^2 + 
(1-t)  \frac{1}{2} \| y - X \hbeta^{(2)} \|_2^2 +
\lambda t \| D \hbeta^{(1)} \|_1  + (1-t) \lambda \| D \hbeta^{(2)} \|_1 \\    
&= t f(\hbeta^{(1)}) + (1-t) f(\hbeta^{(2)}) = f^\star,
\end{align*}
where in the second line we used the strict convexity of the function 
$G(z)=\|y-z\|_2^2$, along with the convexity of $h(z)=\|z\|_1$.  That \smash{$t   
  \hbeta^{(1)} + (1-t) \hbeta^{(2)}$} obtains a lower criterion than $f^\star$
is a contradiction, and this proves part (ii).

Lastly, for part (iii), every solution in the generalized lasso problem
\eqref{eq:genlasso} yields the same fit by part (ii), leading to the same
squared loss; and since every solution also obtains the same (optimal)
criterion value, we conclude that every solution obtains the same penalty
value, provided that $\lambda>0$. 
\end{proof}

Next, we consider the Karush-Kuhn-Tucker (or KKT) conditions to characterize  
optimality of a solution \smash{$\hbeta$} in problem
\eqref{eq:genlasso}. Since there are no contraints, we simply take a
subgradient of the criterion and set it equal to zero. Rearranging gives
\begin{equation}
\label{eq:stat}
X^T (y - X \hbeta) = \lambda D^T \hgamma, 
\end{equation}
where \smash{$\hgamma \in \R^m$} is a subgradient of the $\ell_1$ norm evaluated 
at \smash{$D\hbeta$}, 
\begin{equation}
\label{eq:subg}
\hgamma_i \in \begin{cases}
\{\sign((D \hbeta)_i )\} & \text{if $(D \hbeta )_i \neq 0$} \\   
[-1, 1] & \text{if $(D \hbeta)_i = 0$}
\end{cases}, \quad \text{for $i=1,\ldots,m$}.   
\end{equation}

Since the optimal fit \smash{$X\hbeta$} is unique by Lemma \ref{lem:basic}, the 
left-hand side in \eqref{eq:stat} is always unique.  This immediately leads to
the next result.

\begin{lemma}
\label{lem:gamma}
For any $y,X,D$, and $\lambda>0$, every optimal subgradient \smash{$\hgamma$} in
problem \eqref{eq:genlasso} gives rise to the same value of \smash{$D^T
  \hgamma$}. Moreover, when $D$ has full row rank, the optimal subgradient
\smash{$\hgamma$} is itself unique.  
\end{lemma}

\begin{remark}
When $D$ is row rank deficient, the optimal subgradient \smash{$\hgamma$} is not 
necessarily unique, and thus neither is its associated boundary set (to be
defined in the next subsection).  This complicates the study of uniqueness of
the generalized lasso solution.  In contrast, the optimal subgradient in the
lasso problem is always unique, and its boundary set---called {\it
  equicorrelation set} in this case---is too, which makes the 
study of uniqueness of the lasso solution comparatively simpler
\citep{tibshirani2013lasso}. 
\end{remark}

Lastly, we turn to the dual of problem \eqref{eq:genlasso}.  
Standard arguments in convex analysis, as given in
\citet{tibshirani2011solution}, show that the Lagrangian dual of 
\eqref{eq:genlasso} can be written as\footnote{The form of the 
  dual problem here may superficially appear different from that in 
  \citet{tibshirani2011solution}, but it is equivalent.}
\begin{equation}
\label{eq:dual}
\minimize_{u \in \R^m, \, v \in \R^n} \; \|y-v\|_2^2 \quad\st\quad 
X^T v = D^T u,  \; \|u\|_\infty \leq \lambda.
\end{equation}
Any pair \smash{$(\hu,\hv)$} optimal in the dual \eqref{eq:dual}, and
solution-subgradient pair \smash{$(\hbeta,\hgamma)$} optimal in the primal 
\eqref{eq:genlasso}, \ie, satisfying \eqref{eq:stat}, \eqref{eq:subg}, must
satisfy the primal-dual relationships   
\begin{equation} 
\label{eq:primal_dual}
X\hbeta = y - \hv, \quad\text{and}\quad \hu = \lambda \hgamma.
\end{equation}
We see that \smash{$\hv$}, being a function of the fit
\smash{$X\hbeta$}, is always unique; meanwhile, \smash{$\hu$}, being a function
of the optimal subgradient \smash{$\hgamma$}, is not.  Moreover, the optimality
of \smash{$\hv$} in problem \eqref{eq:dual} can be expressed as
\begin{equation}
\label{eq:dual_proj}
\hv = P_C(y), \quad \text{where} \;\, C = (X^T)^{-1} 
\big(D^T B^m_\infty(\lambda)\big).
\end{equation}
Here, $(X^T)^{-1}(S)$ denotes the preimage of a set $S$ under the linear map 
$X^T$, $D^T S$ denotes the image of a set $S$ under the linear map $D^T$, 
\smash{$B^m_\infty(\lambda) = \{ u \in \R^m : \|u\|_\infty \leq 
  \lambda\}$} is the $\ell_\infty$ ball of radius $\lambda$ in $\R^m$, and 
$P_S(\cdot)$ is the Euclidean projection operator onto a set $S$.
Note that $C$ as defined in \eqref{eq:dual_proj} is a convex polyhedron, 
because the image or preimage of any convex polyhedron under a linear map is a
convex polyhedron.  From \eqref{eq:primal_dual} and \eqref{eq:dual_proj}, we 
may hence write the fit as 
\begin{equation} 
\label{eq:primal_proj}
X\hbeta = (I-P_C)(y),
\end{equation}
the residual from projecting $y$ onto the convex polyhedron $C$.

The conclusion in \eqref{eq:primal_proj}, it turns out, could have been reached
via direction manipulation of the KKT conditions \eqref{eq:stat},
\eqref{eq:subg}, as shown in \citet{tibshirani2012degrees}.  In fact, much of
what can be seen from the dual problem \eqref{eq:dual} can also be
derived using appropriate manipulations of the primal problem
\eqref{eq:genlasso} and its KKT conditions \eqref{eq:stat}, \eqref{eq:subg}.  
However, we feel that the dual perspective, specifically the dual projection in 
\eqref{eq:dual_proj}, offers a simple picture that can be used to
intuitively explain several key results (which might otherwise seem technical 
and complicated in nature).  We will therefore return to it periodically. 

\subsection{Implicit form of solutions}

Fix an arbitrary $\lambda>0$, and let \smash{$(\hbeta,\hgamma)$} denote an 
optimal solution-subgradient pair, \ie, satisfying \eqref{eq:stat},
\eqref{eq:subg}. Following
\citet{tibshirani2011solution,tibshirani2012degrees}, we define the   
{\it boundary set} to contain the indices of components of \smash{$\hgamma$}  
that achieve the maximum possible absolute value,
$$
\cB = \big\{ i \in \{1,\ldots,m\} : | \hgamma_i | = 1 \big\},
$$
and the {\it boundary signs} to be the signs of \smash{$\hgamma$} over the
boundary set, 
$$
s = \sign(\hgamma_\cB).
$$
Since \smash{$\hgamma$} is not necessarily unique, as discussed in the previous
subsection, neither are its associated boundary set and signs $\cB,s$.  Note
that the boundary set contains the {\it active set} 
$$
\cA=\supp(D\hbeta)=\big\{i \in \{1,\ldots,m\} : (D\hbeta)_i \neq  0\big\}
$$
associated with \smash{$\hbeta$}; that $\cB \supseteq \cA$ follows
directly from the property \eqref{eq:subg} (and strict inclusion is certainly 
possible).  Restated, this inclusion tells us that \smash{$\hbeta$} must lie in
the null space of \smash{$D_{-\cB}$}, 
\ie, 
$$
D_{-\cB} \hbeta = 0 \iff \hbeta \in \nul(D_{-\cB}).
$$

Though it seems very simple, the last display provides an avenue for expressing
the generalized lasso fit and solutions in terms of $\cB,s$, which will be quite
useful for establishing sufficient conditions for uniqueness of
the solution.  Multiplying both sides of the stationarity condition
\eqref{eq:stat} by \smash{$P_{\nul(D_{-\cB})}$}, the projection matrix onto  
\smash{$\nul(D_{-\cB})$}, we have
$$
P_{\nul(D_{-\cB})} X^T (y-X\hbeta) = \lambda P_{\nul(D_{-\cB})} D_\cB^T s.
$$
Using \smash{$\hbeta = P_{\nul(D_{-\cB})}\hbeta$}, and solving for the fit 
\smash{$X\hbeta$} (see \citealp{tibshirani2012degrees} for details or the proof
of Lemma \ref{lem:implicit_glm} for the arguments in a more general case) gives  
\begin{equation}
\label{eq:fit}
X\hbeta = XP_{\nul(D_{-\cB})} (XP_{\nul(D_{-\cB})})^+ \big( y- 
\lambda (P_{\nul(D_{-\cB})}X^T)^+ D_\cB^T s\big). 
\end{equation}
Recalling that \smash{$X\hbeta$} is unique from Lemma \ref{lem:basic}, we see
that the right-hand side in \eqref{eq:fit} must agree for all instantiations of
the boundary set and signs $\cB,s$ associated with an optimal subgradient in
problem \eqref{eq:genlasso}.  \citet{tibshirani2012degrees} use this observation
and other arguments to establish an important result that we leverage
later, on the invariance of the space
 \smash{$X \nul(D_{-\cB})=\col(XP_{\nul(D_{-\cB})})$} over all boundary sets
 $\cB$ of optimal subgradients, stated in Lemma \ref{lem:invar} for
 completeness.    

\begin{figure}[htb]
\centering
\includegraphics[width=0.8\textwidth]{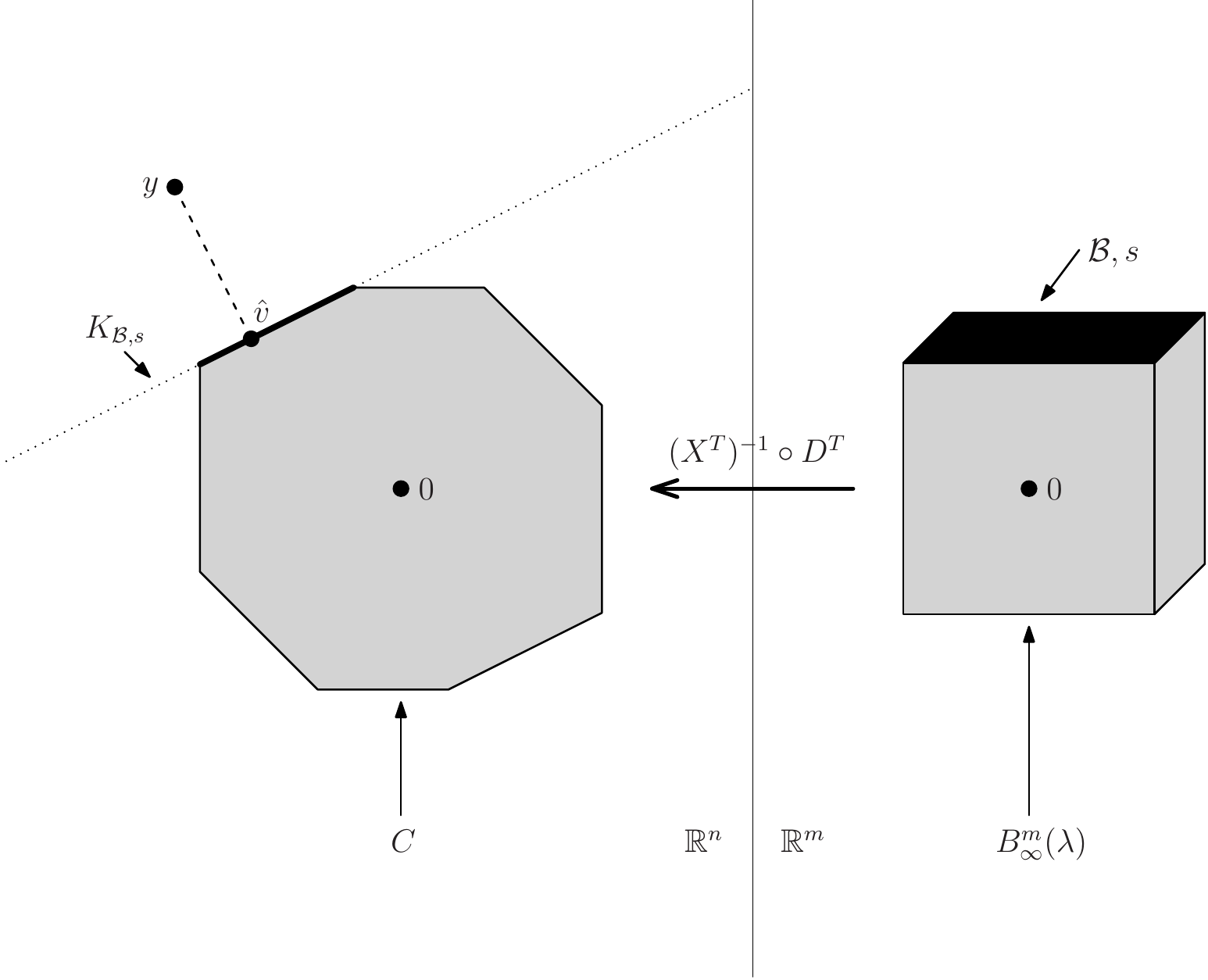}
\caption{\it Geometry of the generalized lasso dual problem \eqref{eq:dual}.
As in \eqref{eq:dual_proj}, the dual solution \smash{$\hv$} may be seen as the
projection of $y$ onto a set $C$, and as in \eqref{eq:primal_proj}, the primal
fit \smash{$X\hbeta$} may be seen as the residual from this projection.
Here, \smash{$C=(X^T)^{-1}(D^T B^m_\infty(\lambda))$}, and as
\smash{$B^m_\infty(\lambda)$} is a polyhedron (and the image or inverse image
of a polyhedron under a linear map is still a polyhedron), $C$ is a
polyhedron as well. This can be used to derive the implicit form \eqref{eq:fit}
for \smash{$X\hbeta$}, based on the face of $C$ on which \smash{$\hv$} lies, as 
explained in Remark \ref{rem:implicit}.}   
\label{fig:geom}
\end{figure}

\begin{remark}
\label{rem:implicit}
As an alternative to the derivation based on the KKT conditions described above, 
the result \eqref{eq:fit} can be argued directly from the geometry
surrounding the dual problem \eqref{eq:dual}.  See Figure \ref{fig:geom} for an
accompanying illustration.  Given that 
\smash{$\hgamma$} has boundary set and signs $\cB,s$, and
\smash{$\hu=\lambda\hgamma$} from \eqref{eq:primal_dual}, 
we see that \smash{$\hu$} must lie on the face of \smash{$B^m_\infty(\lambda)$}
whose affine span is \smash{$E_{\cB,s}=\{u \in \R^m : u_{\cB,s}=\lambda s\}$};
this face is colored in black on the right-hand side of the figure.  Since
\smash{$X^T \hv = D^T \hu$}, this means that \smash{$\hv$} lies on the face of
$C$ whose affine span is \smash{$K_{\cB,s}=(X^T)^{-1} D^T E_{\cB,s}$}; this
face is colored in black on the left-hand side of the figure, and its affine
span \smash{$K_{\cB,s}$} is drawn as a dotted line.  Hence, we may refine our 
view of \smash{$\hv$} in \eqref{eq:dual_proj}, and in turn, \smash{$X\hbeta$} in
\eqref{eq:primal_proj}: namely, we may view \smash{$\hv$} as the projection of
$y$ onto the affine space \smash{$K_{\cB,s}$} (instead of $C$), and the fit 
\smash{$X\hbeta$} as the residual from this affine projection.  
A straightforward 
calculation shows that \smash{$K_{\cB,s} = \lambda (P_{\nul(D_{-\cB})}X^T)^+ D_\cB^T
  s + \nul(P_{\nul(D_{-\cB})}X^T)$}, and another straightforward calculation
shows that the residual from projecting $y$ onto \smash{$K_{\cB,s}$} is
\eqref{eq:fit}.   
\end{remark}

From the expression in \eqref{eq:fit} for the fit \smash{$X\hbeta$}, we also see
that the solution \smash{$\hbeta$} corresponding to the optimal subgradient 
\smash{$\hgamma$} and its boundary set and signs $\cB,s$ must take the form 
\begin{equation}
\label{eq:sol}
\hbeta = (XP_{\nul(D_{-\cB})})^+ \big( y- \lambda (P_{\nul(D_{-\cB})}X^T)^+
D_\cB^T s\big) + b, 
\end{equation}
for some \smash{$b \in \nul(XP_{\nul(D_{-\cB})})$}.  Combining this with
\smash{$b \in \nul(D_{-\cB})$} (following from \smash{$D_{-\cB} \hbeta=0$}), we    
moreover have that \smash{$b \in \nul(X) \cap \nul(D_{-\cB})$}.  In fact,
{\it any} such point \smash{$b \in \nul(X) \cap \nul(D_{-\cB})$} yields a
generalized lasso solution \smash{$\hbeta$} in \eqref{eq:sol} provided that  
$$
s_i \cdot D_i \Big[(XP_{\nul(D_{-\cB})})^+ \big( y- \lambda
(P_{\nul(D_{-\cB})}X^T)^+ D_\cB^T s\big) + b\Big] \geq 0, \quad 
\text{for $i \in \cB$}, 
$$
which says that \smash{$\hgamma$} appropriately matches the signs of the nonzero
components of \smash{$D\hbeta$}, thus \smash{$\hgamma$} remains a proper
subgradient. 

We can now begin to inspect conditions for uniqueness of the generalized lasso
solution.  For a given boundary set \smash{$\cB$} of an optimal subgradient
\smash{$\hgamma$}, 
if we know that \smash{$\nul(X) \cap \nul(D_{-\cB}) = \{0\}$}, 
then there can only be one solution \smash{$\hbeta$} corresponding to
\smash{$\hgamma$} (\ie, such that \smash{$(\hbeta,\hgamma)$} jointly satisfy
\eqref{eq:stat}, \eqref{eq:subg}), and it is given by the expression in
\eqref{eq:sol} with $b=0$.  Further, if we know that \smash{$\nul(X) \cap  
  \nul(D_{-\cB}) = \{0\}$} for {\it all} boundary sets $\cB$ of optimal 
subgradients, and the space \smash{$\nul(D_{-\cB})$} is invariant over all
choices of boundary sets $\cB$ of optimal subgradients, then the right-hand side 
in \eqref{eq:sol} with $b=0$ must agree for all proper instantiations of
$\cB,s$ and it gives the unique generalized lasso solution.  We elaborate on
this in the next section.

\subsection{Invariance of the linear space $X \nul(D_{-\cB})$}  

Before diving into the technical details on conditions for uniqueness in the
next section, we recall a key result from \citet{tibshirani2012degrees}.

\begin{lemma}[Lemma 10 in \citealp{tibshirani2012degrees}]  
\label{lem:invar}
Fix any $X,D$, and $\lambda>0$.  There is a set $\cN
\subseteq \R^n$ of Lebesgue measure zero (that depends on $X,D,\lambda$), such
that for $y \notin \cN$, all boundary sets $\cB$ associated with optimal
subgradients in the generalized lasso problem \eqref{eq:genlasso} give rise to  
the same subspace \smash{$X \nul(D_{-\cB})$}, \ie, there is a single linear
subspace $L \subseteq \R^n$ such that \smash{$L=X \nul(D_{-\cB})$} for all
boundary sets $\cB$ of optimal subgradients.  Moreover, for $y \notin
\cN$, \smash{$L=X \nul(D_{-\cA})$} for all active sets $\cA$ associated
with generalized lasso solutions. 
\end{lemma}

\section{Sufficient conditions for uniqueness}
\label{sec:uniqueness}

\subsection{A condition on certain linear independencies}
\label{sec:uniqueness_rough}

We start by formalizing the discussion on uniqueness in the paragraphs 
proceeding \eqref{eq:sol}.  As before, let $\lambda>0$, and let $\cB$ denote
the boundary set associated with an optimal subgradient in
\eqref{eq:genlasso}. Denote by \smash{$U(\cB) \in \R^{p \times k(\cB)}$} a
matrix with linearly independent columns that span \smash{$\nul(D_{-\cB})$}.  It
is not hard to see that 
$$
\nul(X) \cap \nul(D_{-\cB}) = \{0\} \iff
\nul\big(X U(\cB)\big) = \{0\} \iff 
\rank\big(X U(\cB)\big) = k(\cB).
$$
Let us assign now such a basis matrix \smash{$U(\cB) \in \R^{p \times k(\cB)}$}
to each boundary set $\cB$ corresponding to an optimal subgradient in
\eqref{eq:genlasso}.  We claim that there is a unique generalized lasso
solution, as given in \eqref{eq:sol} with $b=0$, provided that the following two
conditions holds:       
\begin{gather}
\label{eq:cond1}
\rank\big(X U(\cB)\big) = k(\cB) \text{ for all boundary sets $\cB$
  associated with optimal subgradients, and} \\ 
\label{eq:cond2}
\nul(D_{-\cB}) \text{ is invariant across all boundary sets $\cB$
  associated with optimal subgradients}.
\end{gather}
To see this, note that if the space \smash{$\nul(D_{-\cB})$} is invariant
across all achieved boundary sets $\cB$ then so is the matrix
\smash{$P_{\nul(D_{-\cB})}$}.  This, and the fact that
\smash{$P_{\nul(D_{-\cB})} D_B^T s = P_{\nul(D_{-\cB})} D^T \hgamma$} where
\smash{$D^T \hgamma$} is unique from Lemma \ref{lem:gamma}, ensures  
that the right-hand side in \eqref{eq:sol} with $b=0$ agrees no matter the
choice of boundary set and signs \smash{$\cB,s$}.

\begin{remark}
For any subset $\cB \subseteq \{1,\ldots,m\}$, and any matrices 
\smash{$U(\cB), \tilde{U}(\cB) \in \R^{p \times  k(\cB)}$} whose columns 
form a basis for \smash{$\nul(D_{-\cB})$}, it is easy to check that  
\smash{$\rank(XU(\cB))=k(\cB) \iff \rank(X\tilde{U}(\cB))=k(\cB)$}.
Therefore condition \eqref{eq:cond1} is well-defined, \ie, it does not depend on 
the choice of basis matrix $U(\cB)$ associated with \smash{$\nul(D_{-\cB})$} for
each boundary set $\cB$.   
\end{remark}

We now show that, thanks to Lemma \ref{lem:invar}, condition \eqref{eq:cond1} 
(almost everywhere) implies \eqref{eq:cond2}, so the former is alone sufficient
for uniqueness.  

\begin{lemma}
\label{lem:invar_cond}
Fix any $X,D$, and $\lambda>0$. For $y \notin \cN$, where $\cN
\subseteq \R^n$ has Lebesgue measure zero as in Lemma \ref{lem:invar},
condition \eqref{eq:cond1} implies \eqref{eq:cond2}.  Hence, for almost every
$y$, condition \eqref{eq:cond1} is itself sufficient to imply uniqueness of the 
generalized lasso solution.
\end{lemma}

\begin{proof}
Let $y \notin \cN$, and let $L$ be the linear subspace from
Lemma \ref{lem:invar}, \ie, $L=X\nul(D_{-\cB})$ for any boundary set $\cB$
associated with an optimal subgradient in the generalized lasso problem at $y$.
Now fix a particular boundary set $\cB$ associated with an optimal subgradient
and define the linear map $\mathcal{X} : \nul(D_{-\cB}) \to L$ by
$\mathcal{X}(u)=Xu$.  By construction, this map is surjective.  Moreover,
assuming \eqref{eq:cond1}, it is injective, as 
$$
XU(\cB)a = XU(\cB)b \iff X U(\cB) (a-b)=0,
$$
and the right-hand side cannot be true unless $a=b$.  Therefore, $\mathcal{X}$
is bijective and has a linear inverse, and we may write
\smash{$\nul(D_{-\cB})=\mathcal{X}^{-1}(L)$}.  As $\cB$ was arbitrary, this
shows the invariance of \smash{$\nul(D_{-\cB})$} over all proper choices of
$\cB$, whenever $y \notin \cN$.
\end{proof}

From Lemma \ref{lem:invar_cond}, we see that an (almost everywhere) sufficient 
condition for a unique solution in \eqref{eq:genlasso} is that the
vectors $X U_i(\cB) \in \R^n$, $i=1,\ldots,k(\cB)$ are linearly independent, for
all instantiations of boundary sets $\cB$ of optimal  
subgradients.  This may seem a little circular, to give a condition for
uniqueness that itself is expressed in terms of the subgradients of solutions. 
But we will not stop at \eqref{eq:cond1}, and will derive more explicit
conditions on $y,X,D$, and $\lambda>0$ that imply \eqref{eq:cond1} and
therefore uniqueness of the solution in \eqref{eq:genlasso}.  


\subsection{A refined condition on linear independencies}
\label{sec:uniqueness_refined}

The next lemma shows that when condition \eqref{eq:cond1} fails, 
there is a specific type of linear dependence among the columns of $XU(\cB)$,
for a boundary set $\cB$.  The proof is not difficult, but involves careful
manipulations of the KKT conditions \eqref{eq:stat}, and we defer it until the
appendix.     

\begin{lemma}
\label{lem:pre_gp}
Fix any $X,D$, and $\lambda>0$.  Let $y \notin \cN$, the set of zero
Lebesgue measure as in Lemma \ref{lem:invar}. Assume that $\nul(X) \cap \nul(D) 
= \{0\}$, and that the generalized lasso solution is not unique.  Then there
is a pair of boundary set and signs $\cB,s$ corresponding to an optimal 
subgradient in problem \eqref{eq:genlasso}, such that for any matrix
\smash{$U(\cB) \in \R^{p \times k(\cB)}$} whose columns form a basis for  
\smash{$\nul(D_{-\cB})$}, the following property holds of $Z=XU(\cB)$ and 
\smash{$\tilde{s}=U(\cB)^T D_\cB^T s$}: there exist indices $i_1,\ldots,i_k \in   
\{1,\ldots,k(\cB)\}$ with $k \leq n+1$ and \smash{$\tilde{s}_{i_1} \neq 0$},
such that 
\begin{equation}
\label{eq:case1}
Z_{i_2} \in \spa(\{Z_{i_3},\ldots,Z_{i_k}\}),
\end{equation}
when \smash{$\tilde{s}_{i_2}=\cdots=\tilde{s}_{i_k}=0$}, and
\begin{equation}
\label{eq:case2}
Z_{i_1}/\tilde{s}_{i_1} \in \aff(\{Z_{i_j}/\tilde{s}_{i_j} :  \tilde{s}_{i_j}
\neq 0, \, j \geq 2\}) + \spa(\{Z_{i_j} : \tilde{s}_{i_j} = 0\}),
\end{equation}
when at least one of \smash{$\tilde{s}_{i_2},\ldots,\tilde{s}_{i_k}$} is
nonzero. 
\end{lemma}

The spaces on the right-hand sides of both \eqref{eq:case1}, \eqref{eq:case2}
are of dimension at most $n-1$.  To see this, note that
\smash{$\dim(\spa(\{Z_{i_3},\ldots,Z_{i_k}\})) \leq k-2 \leq n-1$}, and also 
$$
\dim\big( \aff(\{Z_{i_j}/\tilde{s}_{i_j} :  \tilde{s}_{i_j}
\neq 0, \, j \geq 2\})\big) +
\dim\big(\spa(\{Z_{i_j} : \tilde{s}_{i_j} = 0\})\big) \leq 
|\mathcal{J}|-2 + |\mathcal{J}^c|=k-2 \leq n-1,
$$
where \smash{$\mathcal{J}=\{ j \in \{1,\ldots,k\} : \tilde{s}_{i_j} \neq 0\}$}.
Hence, because these spaces are at most $(n-1)$-dimensional, neither condition 
\eqref{eq:case1} nor \eqref{eq:case2} should be ``likely'' under a continuous
distribution for the predictor variables $X$.  This is made precise in the next
subsection. 

Before this, we define a deterministic condition on $X$ that ensures
special linear dependencies between the (transformed) columns,
as in \eqref{eq:case1}, \eqref{eq:case2}, never hold. 

\begin{definition}
Fix $D \in \R^{m \times p}$. We say that a matrix $X \in \R^{n \times p}$ is in
{\it $D$-general position} (or {\it $D$-GP}) if the following property holds.
For each subset $\cB \subseteq \{1,\ldots,m\}$ and sign vector \smash{$s \in
  \{-1,1\}^{|\cB|}$}, there is a matrix \smash{$U(\cB) \in \R^{p \times
    k(\cB)}$} whose columns form a basis for \smash{$\nul(D_{-\cB})$}, such that 
for $Z=XU(\cB)$, \smash{$\tilde{s}=U(\cB)^T D_\cB^T s$}, and all $i_1,\ldots,i_k
\in \{1,\ldots,k(\cB)\}$ with \smash{$\tilde{s}_{i_1}\neq 0$} and $k \leq n+1$, 
it holds that
\begin{enumerate}[(i)]
\item $Z_{i_2} \notin \spa(\{Z_{i_3},\ldots,Z_{i_k}\})$, when 
$\tilde{s}_{i_2}=\cdots=\tilde{s}_{i_k}=0$; 
\item $Z_{i_1}/\tilde{s}_{i_1} \notin \aff(\{Z_{i_j}/\tilde{s}_{i_j} :
  \tilde{s}_{i_j} \neq 0, \, j \geq 2\}) +
  \spa(\{Z_{i_j} : \tilde{s}_{i_j} = 0\})$, when at least one of 
$\tilde{s}_{i_2},\ldots,\tilde{s}_{i_k}$ is nonzero.
\end{enumerate}
\end{definition}

\begin{remark}
Though the definition may appear somewhat complicated, a matrix $X$ being in
$D$-GP is actually quite a weak condition, and can hold regardless of the
(relative) sizes of $n,p$.  We will show in the next subsection that it holds
almost surely under an arbitrary continuous probability distribution for the
entries of $X$.  Further, when $X=I$, the above definition essentially
reduces\footnote{We say ``essentially'' here, because our definition of $D$-GP
with $D=I$ allows for a choice of basis matrix $U(\cB)$ for each subset $\cB$,
whereas the standard notion of generally position would mandate (in the notation
of our definition) that $U(\cB)$ be given by the columns of $I$ indexed by
$\cB$.}  to the usual notion of {\it general position} (refer to, \eg,
\citealp{tibshirani2013lasso} for this definition).
\end{remark}

When $X$ is in $D$-GP, we have (by definition) that \eqref{eq:case1},
\eqref{eq:case2} cannot hold for {\it any} $\cB \subseteq \{1,\ldots,m\}$
and $s \in \{-1,1\}^{|\cB|}$ (not just boundary sets and signs); therefore, 
by the contrapositive of Lemma \ref{lem:pre_gp}, if we additionally have
$y \notin \cN$ and $\nul(X) \cap \nul(D) =\{0\}$, then the generalized lasso
solution must be unique.  To emphasize this, we state it as a lemma. 

\begin{lemma}
\label{lem:d_gp}
Fix any $X,D$, and $\lambda>0$.  If $y \notin \cN$, the set of zero
Lebesgue measure as in Lemma \ref{lem:invar}, $\nul(X) \cap \nul(D) =\{0\}$, and
$X$ is in $D$-GP, then the generalized lasso solution is unique.  
\end{lemma}

\subsection{Absolutely continuous predictor variables}
\label{sec:uniqueness_x_cont}

We give an important result that shows the $D$-GP condition is met 
almost surely for continuously distributed predictors.  There are no
restrictions on the relative sizes of $n,p$.  The proof of the next result uses
elementary probability arguments and is deferred until the appendix.

\begin{lemma}
\label{lem:x_cont_gp}
Fix $D \in \R^{m \times p}$, and assume that the entries of $X \in \R^{n \times
  p}$ are drawn from a distribution that is absolutely continuous with respect 
to $(np)$-dimensional Lebesgue measure. Then $X$ is in $D$-GP almost surely.   
\end{lemma}

We now present a result showing that the base condition $\nul(X) \cap \nul(D) =
\{0\}$ is met almost surely for continuously distributed predictors, provided
that $p \leq n$, or $p > n$ and the null space of $D$ is not too large.  Its
proof is elementary and found in the appendix.

\begin{lemma}
\label{lem:x_cont_null}
Fix $D \in \R^{m \times p}$, and assume that the entries of $X \in \R^{n \times 
  p}$ are drawn from a distribution that is absolutely continuous with respect
to $(np)$-dimensional Lebesgue measure.  If either $p \leq n$, or $p>n$ and  
$\nuli(D) \leq n$, then $\nul(X) \cap \nul(D) = \{0\}$ almost surely. 
\end{lemma}

Putting together Lemmas \ref{lem:d_gp}, \ref{lem:x_cont_gp},
\ref{lem:x_cont_null} gives our main result on the uniqueness of the generalized
lasso solution.    

\begin{theorem}
\label{thm:uniqueness}
Fix any $D$ and $\lambda>0$.  Assume the joint distribution of $(X,y)$ is 
absolutely continuous with respect to $(np+n)$-dimensional Lebesgue
measure.  If $p \leq n$, or else $p>n$ and $\nuli(D) \leq n$, then the
solution in the generalized lasso problem \eqref{eq:genlasso} is unique almost
surely.  
\end{theorem}

\begin{remark}
\label{rem:full_row_rank}
If $D$ has full row rank, then by Lemma \ref{lem:gamma} the optimal subgradient
\smash{$\hgamma$} is unique and so the boundary set $\cB$ is also unique.  In
this case, condition \eqref{eq:cond2} is vacuous and condition \eqref{eq:cond1}
is sufficient for uniqueness of the generalized lasso solution for every $y$
(\ie, we do not need to rely on Lemma \ref{lem:invar_cond}, which in turn uses
Lemma \ref{lem:invar}, to prove that \eqref{eq:cond1} is sufficient for almost
every $y$).  Hence, in this case, the condition in Theorem \ref{thm:uniqueness}
that $y|X$ has an absolutely continuous distribution is not needed, and (with
the other conditions in place) uniqueness holds for every $y$, almost surely
over $X$. Under this (slight) sharpening, Theorem \ref{thm:uniqueness} with
$D=I$ reduces to the result in Lemma 4 of \citet{tibshirani2013lasso}. 
\end{remark}

\begin{remark}
Generally speaking, the condition that $\nuli(D) \leq n$ in Theorem
\ref{thm:uniqueness} (assumed in the case $p > n$) is not strong.  In many 
applications of the generalized lasso, the dimension of the null space of $D$ is
small and fixed (\ie, it does not grow with $n$).  For example, recall Corollary
\ref{cor:uniqueness}, where the lower bound $n$ in each of the cases reflects
the dimension of the null space. 
\end{remark}

\subsection{Standardized predictor variables}

A common preprocessing step, in many applications of penalized modeling such as
the generalized lasso, is to {\it standardize} the predictors $X \in \R^{n
  \times p}$, meaning, center each column to have mean 0, and then scale each
column to have norm 1. Here we show that our main uniqueness results carry over,  
mutatis mutandis, to the case of standardized predictor variables.  All proofs
in this subsection are deferred until the appendix.

We begin by studying the case of centering alone.  Let $M = I-\ones\ones^T/n
\in \R^{n \times n}$ be the centering map, and consider the {\it centered
  generalized lasso} problem  
\begin{equation}
\label{eq:genlasso_centered}
\minimize_{\beta \in \R^p} \; \frac{1}{2} \| y - M X \beta \|_2^2 + 
\lambda \| D \beta \|_1. 
\end{equation}
We have the following uniqueness result for centered predictors.

\begin{corollary}
\label{cor:uniqueness_centered}
Fix any $D$ and $\lambda>0$.  Assume the distribution of $(X,y)$ is
absolutely continuous with respect to $(np+n)$-dimensional Lebesgue
measure. If $p \leq n-1$, or $p>n-1$ and $\nuli(D) \leq n-1$, then the
solution in the centered generalized lasso problem \eqref{eq:genlasso_centered}
is unique almost surely.   
\end{corollary}

\begin{remark}
The exact same result as stated in Corollary \ref{cor:uniqueness_centered} holds
for the generalized lasso problem with intercept
\begin{equation}
\label{eq:genlasso_int}
\minimize_{\beta_0 \in \R, \, \beta \in \R^p} \; \frac{1}{2} \| y - \beta_0 \ones - 
X \beta \|_2^2 +  \lambda \| D \beta \|_1. 
\end{equation}
This is because, by minimizing over $\beta_0$ in problem
\eqref{eq:genlasso_int}, we find that this problem is equivalent to minimization
of     
$$
\frac{1}{2} \| M y - M X \beta \|_2^2 +  \lambda \| D \beta \|_1 
$$
over $\beta$, which is just a generalized lasso problem with response
\smash{$V_{-1}^T y$} and predictors \smash{$V_{-1}^T X$}, where the notation
here is as in the proof of Corollary \ref{cor:uniqueness_centered}. 
\end{remark}

Next we treat the case of scaling alone.  Let
\smash{$W_X=\diag(\|X_1\|_2,\ldots,\|X_p\|_2) \in \R^{p \times p}$}, and
consider the {\it scaled generalized lasso} problem
\begin{equation}
\label{eq:genlasso_scaled}
\minimize_{\beta \in \R^p} \; \frac{1}{2} \| y - X W_X^{-1} \beta \|_2^2 +   
\lambda \| D \beta \|_1. 
\end{equation}
We give a helper lemma, on the distribution of a continuous random vector,
post scaling. 

\begin{lemma}
\label{lem:z_scaled}
Let $Z \in \R^n$ be a random vector whose distribution is absolutely continuous 
with respect to $n$-dimensional Lebesgue measure.  Then, the distribution of
$Z/\|Z\|_2$ is absolutely continuous with respect to $(n-1)$-dimensional
Hausdorff measure restricted to the $(n-1)$-dimensional unit sphere,
$\S^{n-1}=\{x \in \R^n : \|x\|_2=1\}$. 
\end{lemma}

We give a second helper lemma, on the $(n-1)$-dimensional Hausdorff measure 
of an affine space intersected with the unit sphere $\S^{n-1}$ (which is
important for checking that the scaled predictor matrix is in $D$-GP, because
here we must check that none of its columns lie in a finite union of affine
spaces).  

\begin{lemma}
\label{lem:s_cap_a}
Let $A \subseteq \R^n$ be an arbitrary affine space, with $\dim(A) \leq n-1$. 
Then $\S^{n-1} \cap A$ has $(n-1)$-dimensional Hausdorff measure zero. 
\end{lemma}

We present a third helper lemma, which establishes that for absolutely
continuous $X$, the scaled predictor matrix \smash{$X W_X^{-1}$} is 
in $D$-GP and satisfies the appropriate null space condition, almost surely.

\begin{lemma}
\label{lem:x_scaled}
Fix $D \in \R^{m \times p}$, and assume that $X \in \R^{n \times
  p}$ has entries drawn from a distribution that is absolutely continuous with
respect to $(np)$-dimensional Lebesgue measure. Then \smash{$X W_X^{-1}$} is in
$D$-GP almost surely. Moreover, if $p \leq n$, or $p > n$ and $\nuli(D) \leq
n$, then \smash{$\nul(X W_X^{-1}) \cap \nul(D) = \{0\}$} almost surely.
\end{lemma}

Combining Lemmas \ref{lem:d_gp}, \ref{lem:x_scaled} gives the following
uniqueness result for scaled predictors.  

\begin{corollary}
\label{cor:uniqueness_scaled}
Fix any $D$ and $\lambda>0$.  Assume the distribution of $(X,y)$ is
absolutely continuous with respect to $(np+n)$-dimensional Lebesgue
measure. If $p \leq n$, or else $p>n$ and $\nuli(D) \leq n$, then the
solution in the scaled generalized lasso problem \eqref{eq:genlasso_scaled} 
is unique almost surely.    
\end{corollary}

Finally, we consider the {\it standardized generalized lasso} problem,
\begin{equation}
\label{eq:genlasso_standard}
\minimize_{\beta \in \R^p} \; \frac{1}{2} \| y - M X W_{MX}^{-1} \beta \|_2^2 +    
\lambda \| D \beta \|_1,
\end{equation}
where, note, the predictor matrix \smash{$MXW_{MX}^{-1}$} has standardized 
columns, \ie, each column has been centered to have mean 0, then scaled to have
norm 1.   We have the following uniqueness result.

\begin{corollary}
\label{cor:uniqueness_standard}
Fix any $D$ and $\lambda>0$.  Assume the distribution of $(X,y)$ is
absolutely continuous with respect to $(np+n)$-dimensional Lebesgue
measure. If $p \leq n-1$, or $p>n-1$ and $\nuli(D) \leq n-1$, then the
solution in the standardized generalized lasso problem
\eqref{eq:genlasso_standard} is unique almost surely.   
\end{corollary}

\section{Smooth, strictly convex loss functions}
\label{sec:general_loss}

\subsection{Generalized lasso with a general loss}

We now extend some of the preceding results beyond the case of squared error
loss, as considered previously.  In particular, we consider the problem 
\begin{equation}
\label{eq:genlasso_g}
\minimize_{\beta \in \R^p} \; G(X\beta; y) + \lambda \| D \beta \|_1,
\end{equation}
where we assume, for each $y \in \R^n$, that the function $G(\,\cdot\,; y)$ is 
{\it essentially smooth} and {\it essentially strictly convex} on $\R^n$.  These
two conditions together mean that $G(\,\cdot\,; y)$ is a closed proper convex
function, differentiable and strictly convex on the interior of its domain
(assumed to be nonempty), with the norm of its gradient approaching $\infty$
along any sequence approaching the boundary of its domain. 
A function that is essentially smooth and essentially strictly convex is 
also called, according to some authors, of {\it Legendre type}; see Chapter
26 of \citet{rockafellar1970convex}.   An important special case of a Legendre
function is one that is differentiable and strictly convex, with full domain  
(all of $\R^n$).     

For much of what follows, we will focus on loss functions of the form 
\begin{equation}
\label{eq:g_glm}
G(z;y) = -y^T z + \psi(z),
\end{equation}
for an essentially smooth and essentially strictly convex function $\psi$ on 
$\R^n$ (not depending on $y$). This is a weak restriction on $G$ and
encompasses, \eg, the cases in which $G$ is the negative log-likelihood
function from a generalized linear model (GLM) for the entries of $y|X$
with a canonical link function, where $\psi$ is the cumulant generating
function. In the case of, say, Bernoulli or Poisson models, this is
$$
G(z;y) = -y^T z + \sum_{i=1}^n \log(1+e^{z_i}), 
\quad \text{or} \quad G(z;y) = -y^T z + \sum_{i=1}^n e^{z_i},
$$
respectively. For brevity, we will often write the loss function as $G(X\beta)$,
hiding the dependence on the response vector $y$.  

\subsection{Basic facts, KKT conditions, and the dual}

The next lemma follows from arguments identical to those for Lemma 
\ref{lem:basic}.

\begin{lemma}
\label{lem:basic_g}
For any $y,X,D,\lambda \geq 0$, and for $G$ essentially smooth and
essentially strictly convex, the following holds of problem
\eqref{eq:genlasso_g}.   
\begin{enumerate}[(i)]
\item There is either zero, one, or uncountably many solutions.
\item Every solution \smash{$\hbeta$} gives rise to the same fitted value 
  \smash{$X \hbeta$}. 
\item If $\lambda>0$, then every solution \smash{$\hbeta$} gives rise to the 
  same penalty value \smash{$\|D\hbeta\|_1$}. 
\end{enumerate}
\end{lemma}

Note the difference between Lemmas \ref{lem:basic_g} and
\ref{lem:basic}, part (i): for an arbitrary (essentially smooth and essentially 
strictly convex) $G$, the criterion in \eqref{eq:genlasso_g} need not attain its
infimum, whereas the criterion in \eqref{eq:genlasso} always does.  This happens 
because the criterion in \eqref{eq:genlasso_g} can have directions of strict
recession (\ie, directions of recession in which the criterion is not constant),
whereas the citerion in \eqref{eq:genlasso} cannot.  Thus in general, 
problem \eqref{eq:genlasso_g} need not have a solution; this is true even in the
most fundamental cases of interest beyond squared loss, \eg, the case of a
Bernoulli negative log-likelihood $G$. Later in Lemma \ref{lem:dual_g}, we
give a sufficient condition for the existence of solutions in
\eqref{eq:genlasso_g}. 

The KKT conditions for problem \eqref{eq:genlasso_g} are 
\begin{equation}
\label{eq:stat_g}
- X^T \nabla G(X \hbeta) = \lambda D^T \hgamma,
\end{equation}
where \smash{$\hgamma \in \R^m$} is (as before) a subgradient of the $\ell_1$
norm evaluated at \smash{$D\hbeta$}, 
\begin{equation}
\label{eq:subg_g}
\hgamma_i \in \begin{cases}
\{\sign((D \hbeta)_i )\} & \text{if $(D \hbeta )_i \neq 0$} \\  
[-1, 1] & \text{if $(D \hbeta)_i = 0$}
\end{cases}, \quad \text{for $i=1,\ldots,m$}.   
\end{equation}
As in the squared loss case, uniqueness of \smash{$X\hbeta$} by
Lemma \ref{lem:basic_g}, along with \eqref{eq:stat_g}, imply the next result.   

\begin{lemma}
\label{lem:gamma_g}
For any $y,X,D,\lambda>0$, and $G$ essentially smooth and essentially 
strictly convex, every optimal subgradient \smash{$\hgamma$} in problem
\eqref{eq:genlasso_g} gives rise to the same value of \smash{$D^T \hgamma$}.  
Furthermore, when $D$ has full row rank, the optimal subgradient
\smash{$\hgamma$} is unique, assuming that problem \eqref{eq:genlasso_g} has a
solution in the first place.
\end{lemma}

Denote by $G^*$ the conjugate function of $G$. When $G$ is essentially 
smooth and essentially strictly convex, the following facts hold (\eg, see 
Theorem 26.5 of \citet{rockafellar1970convex}): 
\begin{itemize}
\item its conjugate $G^*$ is also essentially smooth and essentially strictly
  convex; and 
\item the map $\nabla G : \inte(\dom(G)) \to \inte(\dom(G^*))$ is a
  homeomorphism with inverse $(\nabla G)^{-1} = \nabla G^*$. 
\end{itemize}
The conjugate function is intrinsically tied to duality, directions of
recession, and the existence of solutions. Standard arguments in convex
analysis, deferred to the appendix, give the next result.

\begin{lemma}
\label{lem:dual_g}
Fix any $y,X,D$, and $\lambda \geq 0$.  Assume $G$ is essentially smooth 
and essentially strictly convex.  The Lagrangian dual of problem
\eqref{eq:genlasso_g} can be written as  
\begin{equation}
\label{eq:dual_g}
\minimize_{u \in \R^m, \, v \in \R^u} \; G^*(-v) \quad\st\quad 
X^T v = D^T u,  \; \|u\|_\infty \leq \lambda,
\end{equation}
where $G^*$ is the conjugate of $G$. Any dual optimal pair \smash{$(\hu,\hv)$}
in \eqref{eq:dual_g}, and primal optimal solution-subgradient pair
\smash{$(\hbeta,\hgamma)$} in \eqref{eq:genlasso_g}, \ie, satisfying
\eqref{eq:stat_g}, \eqref{eq:subg_g}, assuming they all exist, must satisfy the
primal-dual relationships    
\begin{equation} 
\label{eq:primal_dual_g}
\nabla G(X\hbeta) = -\hv, \quad\text{and}\quad \hu = \lambda \hgamma.
\end{equation}
Lastly, existence of primal and dual solutions is guaranteed under the
conditions 
\begin{gather}
\label{eq:dom_g}
0 \in \inte(\dom(G)), \\
\label{eq:ran_g}
(-C) \cap \inte(\ran(\nabla G)) \neq \emptyset,
\end{gather}
where \smash{$C=(X^T)^{-1} (D^T B^m_\infty(\lambda))$}. In particular, under 
\eqref{eq:dom_g} and $C \neq \emptyset$, a solution exists in the dual problem
\eqref{eq:dual_g}, and under \eqref{eq:dom_g}, \eqref{eq:ran_g}, a solution
exists in the primal problem \eqref{eq:genlasso_g}. 
\end{lemma}
 
Assuming that primal and dual solutions exist, we see from
\eqref{eq:primal_dual_g} in the above lemma that \smash{$\hv$} must be unique
(by uniqueness of \smash{$X\hbeta$}, from Lemma \ref{lem:basic_g}), but
\smash{$\hu$} need not be (as \smash{$\hgamma$} is not necessarily unique).  
Moreover, under condition \eqref{eq:dom_g}, we know that $G$ is differentiable
at 0, and $\nabla G^*(\nabla G (0)) = 0$, hence we may rewrite \eqref{eq:dual_g}
as       
\begin{equation}
\label{eq:dual_g2}
\minimize_{u \in \R^m, \, v \in \R^n} \; D_{G^*}\big(-v, \nabla G(0) \big)
\quad\st\quad X^T v = D^T u,  \; \|u\|_\infty \leq \lambda, 
\end{equation}
where $D_f(x,z) = f(x)-f(z)-\langle \nabla f(z),x-z \rangle$ denotes the {\it
  Bregman divergence} between points $x,z$, with respect to a function $f$.  
Optimality of \smash{$\hv$} in \eqref{eq:dual_g2} may be expressed as 
\begin{equation}
\label{eq:dual_proj_g}
\hv = -P^{G^*}_{-C}\big(\nabla G(0) \big), \quad \text{where} \;\,  
C = (X^T)^{-1} \big(D^T B^m_\infty(\lambda)\big).
\end{equation}
Here, recall $(X^T)^{-1}(S)$ denotes the preimage of a set $S$ under the linear
map $X^T$, $D^T S$ denotes the image of a set $S$ under the linear map $D^T$,   
\smash{$B^m_\infty(\lambda) = \{ u \in \R^m : \|u\|_\infty \leq 
  \lambda\}$} is the $\ell_\infty$ ball of radius $\lambda$ in $\R^m$, and now
$P^f_S(\cdot)$ is the projection operator onto a set $S$ with respect to the
Bregman divergence of a function $f$, \ie, \smash{$P^f_S(z) = \argmin_{x \in S} 
  D_f(x,z)$}. 
From \eqref{eq:dual_proj_g} and \eqref{eq:primal_dual_g}, we see that 
\begin{equation}
\label{eq:primal_proj_g}
X\hbeta = \nabla G^* \Big( P^{G^*}_{-C}\big(\nabla G(0) \big)\Big).
\end{equation}
We note the analogy between \eqref{eq:dual_proj_g}, \eqref{eq:primal_proj_g}
and \eqref{eq:dual_proj}, \eqref{eq:primal_proj} in the squared loss case; for
\smash{$G(z)=\frac{1}{2} \|y-z\|_2^2$}, we have $\nabla
G(0)=-y$, \smash{$G^*(z)=\frac{1}{2} \|y+z\|_2^2 - \frac{1}{2} \|y\|_2^2$},
\smash{$\nabla G^*(z)=y+z$}, 
\smash{$-P^{G^*}_{-C}(\nabla G(0)) = P_C(y)$}, and so \eqref{eq:dual_proj_g},   
\eqref{eq:primal_proj_g} match \eqref{eq:dual_proj}, \eqref{eq:primal_proj},
respectively.  But when $G$ is non-quadratic, we see that the dual solution
\smash{$\hv$} and primal fit \smash{$X\hbeta$} are given in terms of a
non-Euclidean projection operator, defined with respect to
the Bregman divergence of $G^*$.  See Figure \ref{fig:geom_g} for an
illustration.  This complicates the study of the primal and 
dual problems, in comparison to the squared loss case; still, as we 
will show in the coming subsections, several key properties of primal and
dual solutions carry over to the current general loss setting.  

\begin{figure}[htb]
\centering
\includegraphics[width=0.8\textwidth]{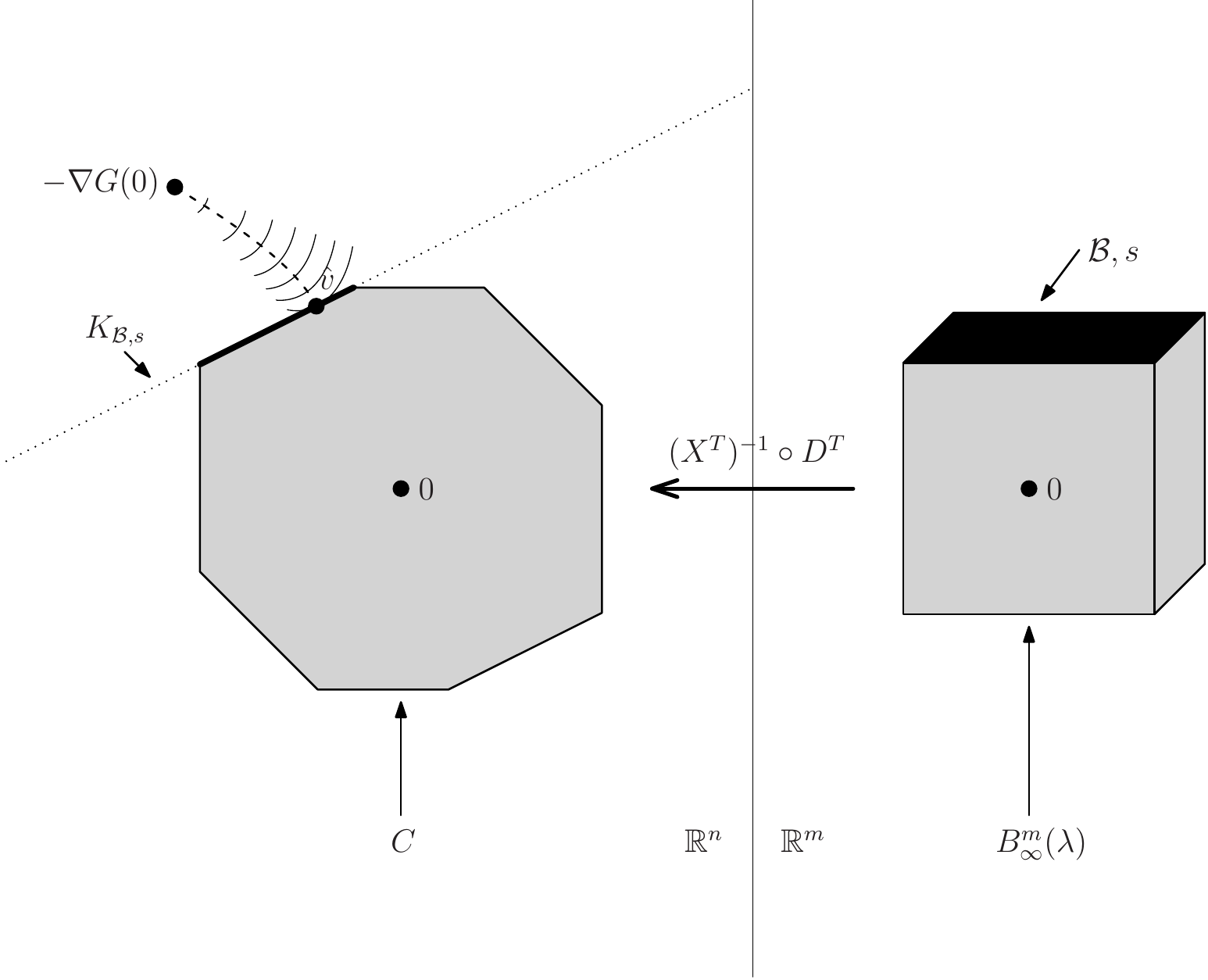}
\caption{\it Geometry of the dual problem \eqref{eq:dual_g2}, for a general loss  
  $G$.  As in \eqref{eq:dual_proj_g}, the dual solution \smash{$\hv$}
  may be seen as the Bregman projection of \smash{$-\nabla G(0)$} onto a set
  $C$ with respect to the map $x \mapsto G^*(-x)$ (where $G^*$ is the
  conjugate of $G$).  Shown in the figure are the contours of this
  map, around \smash{$-\nabla G(0)$}; the point \smash{$\hv$} lies at the
  intersection of the lowest-level contour and $C$.  Here, as in the squared
  loss case, \smash{$C=(X^T)^{-1}(D^T B^m_\infty(\lambda))$}, which is a
  polyhedron. This realization can be used to derive the implicit form
  \eqref{eq:fit_glm} for \smash{$X\hbeta$}, based on \eqref{eq:primal_proj_g}
  and the face of $C$ on which \smash{$\hv$} lies, as explained in Remark
  \ref{rem:implicit_glm}.}     
\label{fig:geom_g}
\end{figure}

\subsection{Existence in (regularized) GLMs}
\label{sec:existence}

Henceforth, we focus on the case in which $G$ takes the form
\eqref{eq:g_glm}. The stationarity condition \eqref{eq:stat_g} is 
\begin{equation}
\label{eq:stat_glm}
X^T \big(y-\nabla \psi(X \hbeta)\big) = \lambda D^T \hgamma,
\end{equation}
and using the identities \smash{$G^*(x)=\psi^*(x+y)$},
\smash{$P^{G^*}_S(x)=P^{\psi^*}_{S+y}(x+y)-y$}, the dual and primal
projections, \eqref{eq:dual_proj_g} and \eqref{eq:primal_proj_g}, become     
\begin{equation}
\label{eq:proj_glm}
\hv = y - P^{\psi^*}_{y-C}\big(\nabla \psi(0) \big), 
\quad \text{and} \quad  
X\hbeta = \nabla \psi^* \Big( P^{\psi^*}_{y-C}\big(\nabla \psi(0)  
\big)\Big). 
\end{equation}
As a check, in the squared loss case, we have
\smash{$\psi(z)=\frac{1}{2} \|z\|_2^2$}, $\nabla\psi(0)=0$,
\smash{$\psi^*(z)=\frac{1}{2} \|z\|_2^2$}, $\nabla\psi^*(z)=z$, 
\smash{$P^{\psi^*}_{y-C}(\nabla \psi(0))=y-P_C(y)$}, so \eqref{eq:proj_glm}
matches \eqref{eq:dual_proj}, \eqref{eq:primal_proj}.  Finally, the conditions
\eqref{eq:dom_g}, \eqref{eq:ran_g} that guarantee the existence of primal and
dual solutions become 
\begin{gather}
\label{eq:dom_psi}
0 \in \inte(\dom(\psi)), \\
\label{eq:ran_psi}
y \in \inte(\ran(\nabla \psi)) + C,
\end{gather}
where recall \smash{$C=(X^T)^{-1} (D^T B^m_\infty(\lambda))$}. 

We take somewhat of a detour from our main goal (establishing
uniqueness in \eqref{eq:genlasso_g}), and study the existence conditions 
\eqref{eq:dom_psi}, \eqref{eq:ran_psi}. To gather insight, we examine them in 
detail for some cases of interest.  We begin by looking at unregularized
($\lambda=0$) logistic and Poisson regression.  The proof of the next result is
straightforward in all but the logistic regression case, and is given in the
appendix. 

\begin{lemma}
\label{lem:exist_glm_unreg}
Fix any $y,X$.  Assume that $G$ is of the form \eqref{eq:g_glm}, where $\psi$ is
essentially smooth and essentially strictly convex, satisfying $0 \in
\inte(\dom(\psi))$.  Consider problem \eqref{eq:genlasso_g}, with
$\lambda=0$. Then the sufficient condition \eqref{eq:ran_psi} for the existence
of a solution is equivalent to  
\begin{equation}
\label{eq:ran_psi_unreg}
y \in \inte(\ran(\nabla \psi)) + \nul(X^T).
\end{equation}
For logistic regression, where \smash{$\psi(z)=\sum_{i=1}^n \log(1+e^{z_i})$}
and $y \in \{0,1\}^n$, if we write $Y_i=2y_i-1 \in \{-1,1\}$,  $i=1,\ldots,n$, 
and we denote by $x_i \in \R^p$, $i=1,\ldots,n$ the rows of $X$, then condition  
\eqref{eq:ran_psi_unreg} is equivalent to  
\begin{equation}
\label{eq:ran_psi_log}
\text{there does not exist $b \neq 0$ such that $Y_i x_i^T b \geq 0$,
  $i=1,\ldots,n$}. 
\end{equation}
For Poisson regression, where \smash{$\psi(z)=\sum_{i=1}^n e^{z_i}$} and $y \in 
\mathbb{N}^n$ (where $\mathbb{N}=\{0,1,2,\ldots\}$ denotes the set of natural 
numbers), condition \eqref{eq:ran_psi_unreg} is equivalent to 
\begin{equation}
\label{eq:ran_psi_pois}
\text{there exists $\delta \in \nul(X^T)$ such that $y_i + \delta_i > 0$,
  $i=1,\ldots,n$}. 
\end{equation}
\end{lemma}

\begin{remark}
For the cases of logistic and Poisson regression, the lemma shows that the
sufficient condition \eqref{eq:ran_psi} for the existence of a solution (note
\eqref{eq:dom_psi} is automatically satisfied, as $\dom(\psi)=\R^n$ in 
these cases) reduces to \eqref{eq:ran_psi_log} and \eqref{eq:ran_psi_pois},
respectively.  Interestingly, in both cases, this recreates a well-known {\it  
  necessary} and sufficient condition for the existence of the maximum
likelihood estimate (MLE); see \citet{albert1984existence} for the logistic 
regression condition \eqref{eq:ran_psi_log}, and \citet{haberman1974log} for the 
Poisson regression condition \eqref{eq:ran_psi_pois}.  The former
condition \eqref{eq:ran_psi_log} is particularly intuitive, and says that the
logistic MLE exists if and only if there is no hyperplane that
``quasicompletely'' separates the points $x_i$, $i=1,\ldots,n$ into the 
positive and negative classes (using the terminology of
\citet{albert1984existence}).  For a modern take on this condition, see  
\citet{candes2018phase}.
\end{remark}

Now we inspect the regularized case ($\lambda>0$).  The proof of the next
result is straightforward and can be found in the appendix.

\begin{lemma}
\label{lem:exist_glm_reg}
Fix any $y,X,D$, and $\lambda>0$. Assume that $G$ is of the form 
\eqref{eq:g_glm}, where we are either in the logistic case,
\smash{$\psi(z)=\sum_{i=1}^n \log(1+e^{z_i})$}  and $y \in \{0,1\}^n$, or in the
Poisson case, \smash{$\psi(z)=\sum_{i=1}^n e^{z_i}$} and $y \in \mathbb{N}^n$
In either case, a sufficient (but not necessary) condition for
\eqref{eq:ran_psi} to hold, and hence for a solution to exist in problem
\eqref{eq:genlasso_g}, is   
\begin{equation}
\label{eq:ran_psi_reg}
\nul(D) \subseteq \nul(X).
\end{equation}
\end{lemma}

\begin{remark}
We note that, in particular, condition \eqref{eq:ran_psi_reg} always holds when
$D=I$, which implies that lasso penalized logistic regression and
lasso penalized Poisson regression always have solutions.  
\end{remark}

\subsection{Implicit form of solutions}

Fix an arbitrary $\lambda>0$, and let \smash{$(\hbeta,\hgamma)$} denote an 
optimal solution-subgradient pair, \ie, satisfying \eqref{eq:stat_g},
\eqref{eq:subg_g}. As before, we define the boundary set and boundary signs in
terms of \smash{$\hgamma$},    
$$
\cB = \big\{ i \in \{1,\ldots,m\} : | \hgamma_i | = 1 \big\},
\quad \text{and} \quad s = \sign(\hgamma_\cB).
$$
and the active set and active signs in terms of \smash{$\hbeta$}, 
$$
\cA=\supp(D\hbeta)=\big\{i \in \{1,\ldots,m\} : (D\hbeta)_i \neq  0\big\},
\quad \text{and} \quad r = \sign(\hgamma_\cA).
$$
By \eqref{eq:stat_g}, we have that $\cA \subseteq \cB$.  
In general, $\cA,r,\cB,s$ are not unique, as neither \smash{$\hbeta$}
nor \smash{$\hgamma$} are. 

The next lemma gives an implicit form for the fit and solutions in
\eqref{eq:genlasso_g}, with $G$ as in \eqref{eq:g_glm}, akin to the results
\eqref{eq:fit}, \eqref{eq:sol} in the squared loss case.  Its proof stems
directly from the KKT conditions \eqref{eq:stat_glm}; it is somewhat technical
and deferred until the appendix.  

\begin{lemma}
\label{lem:implicit_glm}
Fix any $y,X,D$, and $\lambda>0$.  Assume that $G$ is of the form
\eqref{eq:g_glm}, where $\psi$ is essentially smooth and essentially  
strictly convex, and satisfies \eqref{eq:dom_psi}, \eqref{eq:ran_psi}.
Let \smash{$\hbeta$} be a solution in problem \eqref{eq:genlasso_g}, and let 
\smash{$\hgamma$} be a corresponding optimal subgradient, with boundary set 
and boundary signs $\cB,s$. Define the affine subspace
\begin{equation}
\label{eq:k}
K_{\cB,s} = \lambda (P_{\nul(D_{-\cB})}X^T)^+ D_\cB^T s +
\nul(P_{\nul(D_{-\cB})}X^T). 
\end{equation}
Then the unique fit can be expressed as  
\begin{equation}
\label{eq:fit_glm}
X\hbeta = \nabla \psi^* \Big( P^{\psi^*}_{y-K_{\cB,s}}\big(\nabla \psi(0)   
\big)\Big),
\end{equation}
and the solution can be expressed as
\begin{equation}
\label{eq:sol_glm}
\hbeta = (XP_{\nul(D_{-\cB})})^+ 
\nabla \psi^* \Big( P^{\psi^*}_{y-K_{\cB,s}}\big(\nabla \psi(0) \big)\Big) + b,
\end{equation}
for some \smash{$b \in \nul(X) \cap \nul(D_{-\cB})$}.  Similarly, letting
$\cA,r$ denote the active set and active signs of \smash{$\hbeta$}, the same
expressions hold as in the last two displays with $\cB,s$ replaced by $\cA,r$
(\ie, with the affine subspace of interest now being
\smash{$K_{\cA,r} = \lambda (P_{\nul(D_{-\cA})}X^T)^+ D_\cA^T r +
\nul(P_{\nul(D_{-\cA})}X^T)$}). 
\end{lemma}

\begin{remark}
\label{rem:implicit_glm}
The proof of Lemma \ref{lem:implicit_glm} derives the representation
\eqref{eq:fit_glm} using technical manipulation of the KKT conditions.  But the
same result can be derived using the geometry  
surrounding the dual problem \eqref{eq:dual_g2}. See Figure \ref{fig:geom_g} for
an accompanying illustration, and Remark \ref{rem:implicit} for a similar
geometric argument in the squared loss case.  As
\smash{$\hgamma$} has boundary set and signs $\cB,s$, and
\smash{$\hu=\lambda\hgamma$} from \eqref{eq:primal_dual_g}, 
we see that \smash{$\hu$} must lie on the face of \smash{$B^m_\infty(\lambda)$}
whose affine span is \smash{$E_{\cB,s}=\{u \in \R^m : u_{\cB,s}=\lambda s\}$};
and as \smash{$X^T \hv = D^T \hu$}, we see that \smash{$\hv$} lies on the 
face of $C$ whose affine span is \smash{$K_{\cB,s}=(X^T)^{-1} D^T E_{\cB,s}$},
which, it can be checked, can be rewritten explicitly as the affine subspace in
\eqref{eq:k}.  Hence, the projection of $\nabla G(0)$ onto $-C$ lies on a face 
whose affine span is \smash{$-K_{\cB,s}$}, and we can write   
$$
-\hv = P^{G^*}_{-K_{\cB,s}}\big(\nabla G(0) \big),
$$
\ie, we can simply replace the set $-C$ in \eqref{eq:dual_proj_g} with
\smash{$-K_{\cB,s}$}.  When $G$ is of the form \eqref{eq:g_glm}, repeating the
same arguments as before therefore shows that the dual and primal projections in  
\eqref{eq:proj_glm} hold with $-C$ replaced by \smash{$-K_{\cB,s}$}, which
yields the primal projection result in \eqref{eq:fit_glm} in the lemma. 
\end{remark}

Though the form of solutions in \eqref{eq:sol_glm} appears more complicated in
form than the form \eqref{eq:sol} in the squared loss case, we see that one
important property has carried over to the general loss setting, namely, the
property that \smash{$b \in \nul(X) \cap \nul(D_{-\cB})$}.  As before, let
us assign to each boundary set $\cB$ associated with an optimal subgradient in 
\eqref{eq:genlasso_g} a basis matrix \smash{$U(\cB) \in \R^{p \times k(\cB)}$},
whose linearly independent columns that span \smash{$\nul(D_{-\cB})$}.  Then by
the same logic as explained at the beginning of Section
\ref{sec:uniqueness_rough}, we see that, under the conditions of Lemma
\ref{lem:implicit_glm}, there is a unique solution in \eqref{eq:genlasso_g},
given by \eqref{eq:sol_glm} with $b=0$, provided that conditions
\eqref{eq:cond1}, \eqref{eq:cond2} hold. 

The arguments in the squared loss case, proceeding the observation
of \eqref{eq:cond1}, \eqref{eq:cond2} as a sufficient condition, relied on the
invariance of the linear subspace \smash{$X\nul(D_{-\cB})$} over all boundary
sets $\cB$ of optimal subgradients in the generalized lasso problem
\eqref{eq:genlasso}.  This key result was established, recall, in Lemma 10 of
\citet{tibshirani2012degrees}, transcribed in our Lemma \ref{lem:invar} for
convenience.  For the general loss setting, no such invariance result exists (as
far as we know). Thus, with uniqueness in mind as the end goal, we take somewhat
of a detour and study local properties of generalized lasso solutions, and
invariance of the relevant linear subspaces, over the next two subsections.

\subsection{Local stability}

We establish a result on the local stability of the boundary set and boundary
signs $\cB,s$ associated with an optimal solution-subgradient pair
\smash{$(\hbeta,\hgamma)$}, \ie, satisfying \eqref{eq:stat_g},
\eqref{eq:subg_g}.  This is a generalization of Lemma 9 in
\citet{tibshirani2012degrees}, which gives the analogous result for the case of 
squared loss.  We must first introduce some notation.  For arbitrary subsets
$\cA \subseteq  \cB \subseteq \{1,\ldots,m\}$, denote
\begin{equation}
\label{eq:m}
M_{\cA,\cB} = P_{[D_{\cB \setminus \cA}(\nul(X) \cap \nul(D_{-\cB}))]^\perp}  
D_{\cB \setminus\cA} (X P_{\nul(D_{-\cB})})^+.
\end{equation}
(By convention, when $\cA=\cB$, we set \smash{$M_{\cA,\cB}=0$}.) Define  
\begin{equation}
\label{eq:n}
\cN = \bigcup_{\substack{\cA,\cB,s : \\ M_{\cA,\cB} \neq 0}} 
\Big(K_{\cB,s} + \nabla\psi\big(\col(XP_{\nul(D_{-\cB})}) \cap
\nul(M_{\cA,\cB})\big) \Big).
\end{equation}
The union above is taken over all subsets $A \subseteq \cB \subseteq  
\{1,\ldots,m\}$ and vectors \smash{$s \in \{-1,1\}^{|\cB|}$}, such that 
\smash{$M_{\cA,\cB} \neq 0$}; and \smash{$K_{\cB,s},M_{\cA,\cB}$}, are as
defined in \eqref{eq:k}, \eqref{eq:m}, respectively. We use somewhat of an abuse
in notation in writing \smash{$\nabla\psi(\col(XP_{\nul(D_{-\cB})}) \cap
  \nul(M_{\cA,\cB}))$}; for an arbitrary triplet $(\cA,\cB,s)$, of course,
$\col(XP_{\nul(D_{-\cB})}) \cap \nul(M_{\cA,\cB})$ need not be contained in
$\inte(\dom(\psi))$, and so really, each such term in the above union should be
interpreted as \smash{$\nabla\psi(\col(XP_{\nul(D_{-\cB})}) \cap  
  \nul(M_{\cA,\cB}) \cap \inte(\dom(\psi)))$}.

Next we present the local stability result.  Its proof is lengthy and deferred
until the appendix.  

\begin{lemma}
\label{lem:local_glm}
Fix any $X,D$, and $\lambda>0$.  Fix $y \notin \cN$, where the set $\cN$ is
defined in \eqref{eq:n}.  Assume that $G$ is of the form \eqref{eq:g_glm}, where 
$\psi$ is essentially smooth and essentially strictly convex, satisfying
\eqref{eq:dom_psi}, \eqref{eq:ran_psi}.  That is, our assumptions on the
response are succinctly: \smash{$y \in \cN^c \cap (\inte(\ran(\nabla
  \psi))+C)$}.  Denote an optimal solution-subgradient pair in problem 
\eqref{eq:genlasso_g} by \smash{$(\hbeta(y),\hgamma(y))$}, our notation here
emphasizing the dependence on $y$, and similarly, denote the associated 
boundary set, boundary signs, active set, and active signs by 
\smash{$\cB(y),s(y),\cA(y),r(y)$}, respectively.   
There is a neighborhood $U$ of $y$ such that, for any $y' \in U$, problem
\eqref{eq:genlasso_g} has a solution, and in particular, it has an optimal 
solution-subgradient pair \smash{$(\hbeta(y'),\hgamma(y'))$} with the same
boundary set \smash{$\cB(y')=\cB(y)$}, boundary signs \smash{$s(y')=s(y)$},
active set \smash{$\cA(y')=\cA(y)$}, and active signs \smash{$r(y')=r(y)$}.  
\end{lemma}

\begin{remark}
The set $\cN$ defined in \eqref{eq:n} is bigger than it needs to be;
to be precise, the same result as in Lemma \ref{lem:local_glm} actually holds
with $\cN$ replaced by the smaller set
\begin{equation}
\label{eq:n2}
\cN^* = \bigcup_{\substack{\cA,\cB,s : \\ M_{\cA,\cB} \neq 0}} 
\bigg\{ z \in \R^n : 
\nabla \psi^* \Big( P^{\psi^*}_{z-K_{\cB,s}}\big(\nabla \psi(0) \big)\Big)
\in \nul(M_{\cA,\cB}) \bigg\}.
\end{equation}
which can be seen from the proof of Lemma \ref{lem:local_glm}, as can be $\cN^*
\subseteq \cN$.  However, the definition of $\cN$ in \eqref{eq:n} is more
explicit than that of $\cN^*$ in \eqref{eq:n2}, so we stick with the former set
for simplicity.  
\end{remark}

\begin{remark}
\label{rem:n}
For each triplet $\cA,\cB,s$ in the definition 
\eqref{eq:n} over which the union is defined, the sets 
\smash{$K_{\cB,s}$} and \smash{$\col(XP_{\nul(D_{-\cB})}) \cap
  \nul(M_{\cA,\cB})$} both have Lebesgue measure zero, as they are affine    
spaces of dimension at most $n-1$.   When $\nabla \psi : \inte(\dom(\psi)) \to
\inte(\dom(\psi^*))$ is a $C^1$ diffeomorphism---this is true when $\psi$  
is the cumulant generating function for the Bernoulli or Poisson cases---the
image \smash{$\nabla \psi(\col(XP_{\nul(D_{-\cB})}) \cap \nul(M_{\cA,\cB}))$}
also has Lebesgue measure zero, for each triplet $\cA\,\cB,s$, and thus $\cN$
(being a finite union of measure zero sets) has measure zero.   
\end{remark}

\subsection{Invariance of the linear space $X\nul(D_{-\cB})$}

We leverage the local stability result from the last subsection to
establish an invariance of the linear subspace \smash{$X\nul(D_{-\cB})$} over
all choices of boundary sets $\cB$ corresponding to an optimal subgradient in
\eqref{eq:genlasso_g}.  This is a generalization of Lemma 10 in problem
\citet{tibshirani2012degrees}, which was transcribed in our Lemma
\ref{lem:invar}.  The proof is again deferred until the appendix.

\begin{lemma}
\label{lem:invar_glm}
Assume the conditions of Lemma \ref{lem:local_glm}. Then 
all boundary sets $\cB$ associated with optimal subgradients in 
problem \eqref{eq:genlasso_g} give rise to 
the same subspace \smash{$X \nul(D_{-\cB})$}, \ie, there is a single linear
subspace $L \subseteq \R^n$ such that \smash{$L=X \nul(D_{-\cB})$} for all
boundary sets $\cB$ of optimal subgradients.  Further, \smash{$L=X
  \nul(D_{-\cA})$} for all active sets $\cA$ associated with solutions in
\eqref{eq:genlasso_g}. 
\end{lemma}

As already mentioned, Lemmas \ref{lem:local_glm} and \ref{lem:invar_glm} extend
Lemmas 9 and 10, respectively, of \citet{tibshirani2012degrees} to the case of a
general loss function $G$, taking the generalized linear model form in
\eqref{eq:g_glm}. This represents a significant advance in our understanding of
the local nature of generalized lasso solutions outside of the squared loss
case.  For example, even for the special case $D=I$, that logistic lasso
solutions have locally constant active sets, and that $\col(X_A)$ is invariant 
to all choices of active set $A$, provided $y$ is not in an ``exceptional set''
$\cN$, seem to be interesting and important findings.  These results could be
helpful, \eg, in characterizing the divergence, with respect to $y$, of the
generalized lasso fit in \eqref{eq:fit_glm}, an idea that we leave to future
work. 

\subsection{Sufficient conditions for uniqueness}

We are now able to build on the invariance result in Lemma
\ref{lem:invar_glm}, just as we did in the squared loss case, to derive our main
result on uniqueness in the current general loss setting. 

\begin{theorem}
\label{thm:uniqueness_glm}
Fix any $X,D$, and $\lambda>0$.   Assume that $G$ is of the form
\eqref{eq:g_glm}, where $\psi$ is essentially smooth and essentially  
strictly convex, and satisfies \eqref{eq:dom_psi}.  Assume:
\begin{enumerate}[(a)]
\item $\nul(X) \cap \nul(D) =\{0\}$, and $X$ is in $D$-GP; or
\item the entries of $X$ are drawn from a distribution that is absolutely
  continuous on $\R^{np}$, and $p \leq n$; or
\item the entries of $X$ are drawn from a distribution that is absolutely
  continuous on $\R^{np}$, $p > n$, and $\nuli(D) \leq n$.
\end{enumerate}
In case (a), the following holds deterministically, and in cases (b) or (c), it 
holds with almost surely with respect to the distribution of $X$: for any
\smash{$y \in \cN^c \cap (\inte(\ran(\nabla \psi))+C)$}, where $\cN$ is as 
defined in \eqref{eq:n}, problem \eqref{eq:genlasso_g} has a unique solution.  
\end{theorem}

\begin{proof}
Under the conditions of the theorem, Lemma \ref{lem:implicit_glm} shows that any
solution in \eqref{eq:genlasso_g} must take the form \eqref{eq:sol_glm}.  As in
the arguments in Section \ref{sec:uniqueness_rough}, in the squared loss case,
we see that \eqref{eq:cond1}, \eqref{eq:cond2} are together sufficient for
implying uniqueness of the solution in \eqref{eq:genlasso_g}.  Moreover, 
Lemma \ref{lem:invar_glm} implies the linear subspace
\smash{$L=X\nul(D_{-\cB})$} is invariant under all choices of boundary sets
$\cB$ corresponding to optimal subgradients in \eqref{eq:genlasso_g}; as in the
proof of Lemma \ref{lem:invar_cond} in the squared loss case, such invariance
implies that \eqref{eq:cond1} is by itself a sufficient condition.  Finally, if
\eqref{eq:cond1} does not hold, then $X$ cannot be in $D$-GP, which follows by
the applying the arguments Lemma \ref{lem:pre_gp} in the squared loss case to
the KKT conditions \eqref{eq:stat_glm}.  This completes the proof under
condition (a).  Recall, conditions (b) or (c) simply imply (a) by Lemmas
\ref{lem:x_cont_gp} and \ref{lem:x_cont_null}.  
\end{proof}

As explained in Remark \ref{rem:n}, the set $\cN$ in \eqref{eq:n} has Lebesgue
measure zero for $G$ as in \eqref{eq:g_glm}, when $\nabla \psi$ is a 
$C^1$ diffeomorphism, which is true, \eg, for $\psi$ the Bernoulli
or Poisson cumulant generating function. However, in the case that $\psi$ is the 
Bernoulli cumulant generating function, and $G$ is the associated negative
log-likelihood, it would of course be natural to assume that the entries of
$y|X$ follow a Bernoulli distribution, and under this assumption it is not
necessarily true that the event $y \in \cN$ has zero probability.  A similar
statement holds for the Poisson case.  Thus, it does not seem
straightforward to bound the probability that $y \in \cN$ in cases of
fundamental interest, \eg, when the entries of $y|X$ follow a Bernoulli or
Poisson model and $G$ is the associated negative log-likehood, but intuitively
$y \in \cN$ seems ``unlikely'' in these cases.  A careful analysis is left to
future work.

\section{Discussion}
\label{sec:discussion}

In this paper, we derived sufficient conditions for the generalized lasso
problem \eqref{eq:genlasso} to have a unique solution, which allow for $p>n$ 
(in fact, allow for $p$ to be arbitrarily larger than $n$): as long as the
predictors and response jointly follow a continuous distribution, and the null
space of the penalty matrix has dimension at most $n$, our main result in
Theorem \ref{thm:uniqueness} shows that the solution is unique.  We have also 
extended our study to the problem \eqref{eq:genlasso_g}, where the loss is of
generalized linear model form \eqref{eq:g_glm}, and established an analogous
(and more general) uniqueness result in Theorem \ref{thm:uniqueness_glm}.  Along
the way, we have also shown some new results on the local stability of boundary
sets and active sets, in Lemma \ref{lem:local_glm}, and on the invariance of
key linear subspaces, in Lemma \ref{lem:invar_glm}, in the generalized linear
model case, which may be of interest in their own right.   

An interesting direction for future work is to carefully bound the probability
that $y \in \cN$, where $\cN$ is as in \eqref{eq:n}, in some typical generalized   
linear models like the Bernoulli and Poisson cases.  This would give us a 
more concrete probabillistic statement about uniqueness in such cases, following
from Theorem \ref{thm:uniqueness_glm}.  Another interesting direction is to
inspect the application of Theorems \ref{thm:uniqueness} and
\ref{thm:uniqueness_glm} to additive trend filtering and varying-coefficient
models.  Lastly, the local stability result in Lemma \ref{lem:local_glm} seems
to suggest that a nice expression for the divergence of the fit
\eqref{eq:fit_glm}, as a function of $y$, may be possible (furthermore, Lemma 
\ref{lem:invar_glm} suggests that this expression should be invariant to the
choice of boundary set). This may prove useful for various purposes, \eg,
for constructing unbiased risk estimates in penalized generalized linear
models. 

\subsection*{Acknowledgements}

The authors would like to thank Emmanuel Candes and Kevin Lin for several
helpful conversations, that led to the more careful inspection of the existence 
conditions for logistic and Poisson regression, in Section \ref{sec:existence}.      

\appendix
\section{Proofs}

\subsection{Proof of Lemma \ref{lem:pre_gp}}

As the generalized lasso solution is not unique, we know that condition
\eqref{eq:cond1} cannot hold, and there exist $\cB,s$ associated with an
optimal subgradient in problem \eqref{eq:genlasso} for which $\rank(X U(\cB)) <
k(\cB)$, for any \smash{$U(\cB) \in \R^{p \times k(\cB)}$} whose linearly
independent columns span \smash{$\nul(D_{-\cB})$}.  Thus, fix an arbitrary
choice of basis matrix $U(\cB)$. Then by construction we have that 
$Z_i = X U_i(\cB) \in \R^n$, $i=1,\ldots,k(\cB)$ are linearly dependent.       

Note that multiplying both sides of the KKT conditions \eqref{eq:stat} by 
$U(\cB)^T$ gives
\begin{equation}
\label{eq:u_stat}
U(\cB)^T X^T (y-X \hbeta) = \tilde{s},
\end{equation}
by definition of \smash{$\tilde{s}$}. We will first show that the assumptions in 
the lemma, \smash{$\tilde{s} \neq 0$}.  To see this, if 
\smash{$\tilde{s} = 0$}, then at any solution \smash{$\hbeta$} as in
\eqref{eq:sol} associated with $\cB,s$,   
$$
\|D\hbeta\|_1 = \|D_\cB \hbeta\|_1 = s^T D_\cB \hbeta = 0,
$$
since \smash{$\hbeta \in \col(U(\cB))$}.
Uniqueness of the penalty value as in Lemma \ref{lem:basic} now implies that 
\smash{$\|D\hbeta\|_1=0$} at {\it all} generalized lasso solutions (not only
those stemming from $\cB,s$).  Nonuniqueness of the solution is therefore only
possible if \smash{$\nul(X) \cap \nul(D) \neq \{0\}$}, contradicting the 
setup in the lemma.  

We may now choose $i_1 \in \{1,\ldots,k(\cB)\}$ such that
\smash{$\tilde{s}_{i_1} \neq 0$}, and $i_2,\ldots,i_k \in \{1,\ldots,k(\cB)\}$
such that $k \leq n+1$ and  
\begin{equation}
\label{eq:z_sum}
\sum_{j=1}^k c_j Z_{i_j} = 0.
\end{equation}
for some $c \neq 0$.  Taking an inner product on both sides with the residual 
\smash{$y-X\hbeta$}, and invoking the modified KKT conditions \eqref{eq:u_stat},  
gives  
\begin{equation}
\label{eq:c_sum}
\sum_{j=1}^k c_j \tilde{s}_{i_j} = 0.
\end{equation}
There are two cases to consider.  If \smash{$\tilde{s}_{i_j} = 0$} for
all $j=2,\ldots,k$, then we must have $c_1=0$, so from \eqref{eq:z_sum}, 
\begin{equation}
\label{eq:pre_case1} 
\sum_{j=2}^k c_j Z_{i_j} = 0.
\end{equation}
If instead \smash{$\tilde{s}_{i_j} \neq 0$} for some $j=2,\ldots,k$, then
define \smash{$\mathcal{J}=\{ j \in \{1,\ldots,k\} : \tilde{s}_{i_j} \neq 0\}$} 
(which we know in the present case has cardinality $|\mathcal{J}| \geq 2$). 
Rewrite \eqref{eq:c_sum} as
$$
c_1 \tilde{s}_{i_1} = -\sum_{j \in \mathcal{J} \setminus \{1\}}  
c_j \tilde{s}_{i_j},
$$
and hence rewrite \eqref{eq:z_sum} as
$$
\sum_{j \in \mathcal{J}} c_j \tilde{s}_{i_j} \frac{Z_{i_j}}{\tilde{s}_{i_j}} + 
\sum_{j \notin \mathcal{J}} c_j Z_{i_j} = 0,
$$
or
$$
\frac{Z_{i_1}}{\tilde{s}_{i_1}} = \frac{-1}{c_1 \tilde{s}_{i_1}} 
\sum_{j \in \mathcal{J} \setminus \{1\}} c_j \tilde{s}_{i_j}
\frac{Z_{i_j}}{\tilde{s}_{i_j}} +  
\frac{-1}{c_1 \tilde{s}_{i_1}} \sum_{j \notin \mathcal{J}} c_j Z_{i_j}.  
$$
or letting \smash{$a_{i_j}=-c_j\tilde{s}_{i_j}/(c_1\tilde{s}_{i_1})$} for $j \in 
  \mathcal{J}$,  
\begin{equation}
\label{eq:pre_case2}
\frac{Z_{i_1}}{\tilde{s}_{i_1}} =
\sum_{j \in \mathcal{J} \setminus \{1\}} a_{i_j} \frac{Z_{i_j}}{\tilde{s}_{i_j}}
+ \frac{-1}{c_1 \tilde{s}_{i_1}} \sum_{j \notin \mathcal{J}} c_j Z_{i_j}, 
\quad \text{where $\sum_{j \in \mathcal{J} \setminus \{1\}} a_{i_j}=1$}. 
\end{equation}
Reflecting on the two conclusions \eqref{eq:pre_case1}, \eqref{eq:pre_case2} 
from the two cases considered, we can reexpress these as \eqref{eq:case1}, 
\eqref{eq:case2}, respectively, completing the proof.
\hfill\qedsymbol

\subsection{Proof of Lemma \ref{lem:x_cont_gp}}

Fix an arbitrary $\cB \subseteq \{1,\ldots,m\}$ and $s \in \{-1,1\}^{|\cB|}$.
Define \smash{$U(\cB) \in \R^{p \times k(\cB)}$} whose columns form a basis for
\smash{$\nul(D_{-\cB})$} by running Gauss-Jordan elimination on $D_{-\cB}$.  We
may assume without a loss of generality that this is of the form 
$$
U(\cB) = \left[\begin{array}{c} 
I \\ F \end{array} \right],
$$
where \smash{$I \in \R^{k(\cB) \times k(\cB)}$} is the identity matrix and 
\smash{$F \in \R^{(p-k(\cB)) \times k(\cB)}$} is a generic dense matrix.  (If
need be, then we can always permute the columns of $X$, \ie, relabel the
predictor variables, in order to obtain such a form.) This allows us to express
the columns of $Z=XU(\cB)$ as  
$$
Z_i = \sum_{\ell=1}^p X_\ell U_{\ell i}(\cB) = X_i +  
\sum_{\ell=1}^{p-k(\cB)} X_{\ell+k(\cB)} F_{\ell i}, \quad  
\text{for $i=1,\ldots,k(\cB)$}.  
$$
Importantly, for each $i=1,\ldots,k(\cB)$, we see that only $Z_i$ depends
on $X_i$ (\ie, no other $Z_j$, $j \neq i$ depends on $X_i$).  Select
any $i_1,\ldots,i_k \in \{1,\ldots,k(\cB)\}$ with \smash{$\tilde{s}_{i_1}
  \neq 0$} and $k \leq n+1$.  Suppose first that
\smash{$\tilde{s}_{i_2}=\cdots=\tilde{s}_{i_k}=0$}.  Then
$$
Z_{i_2} \in \spa(\{Z_{i_3},\ldots,Z_{i_k}\}) \iff 
X_{i_2} \in -\sum_{\ell=1}^{p-k(\cB)} X_{\ell+k(\cB)} F_{\ell i} + 
\spa(\{Z_{i_3},\ldots,Z_{i_k}\}).
$$
Conditioning on $X_j$, $j \neq i_2$, the right-hand side above is just some
fixed affine space of dimension at most $n-1$, and so 
$$
\P\bigg( X_{i_2} \in
-\sum_{\ell=1}^{p-k(\cB)} X_{\ell+k(\cB)} F_{\ell i} + 
\spa(\{Z_{i_3},\ldots,Z_{i_k}\}) \, \bigg|\, X_j, j \neq i_2 \bigg) = 0, 
$$
owing to the fact that $X_{i_2} \,|\, X_j, j \neq i_2$ has a continuous
distribution over $\R^n$. Integrating out over $X_j$, $j \neq i_2$ then gives  
$$
\P\bigg( X_{i_2} \in
-\sum_{\ell=1}^{p-k(\cB)} X_{\ell+k(\cB)} F_{\ell i} + 
\spa(\{Z_{i_3},\ldots,Z_{i_k}\}) \bigg) = 0,
$$
which proves a violation of case (i) in the definition of $D$-GP happens
with probability zero.  Similar arguments show that a violation of case (ii) in
the definition of $D$-GP happens with probability zero.  Taking a union bound 
over all possible $\cB,s,i_1,\ldots,i_k$, and $k$ shows that any violation of the
defining properties of the $D$-GP condition happens with probability zero,
completing the proof.
\hfill\qedsymbol 

\subsection{Proof of Lemma \ref{lem:x_cont_null}}

Checking that $\nul(X) \cap \nul(D) = \{0\}$ is equivalent to checking that
the matrix 
$$
M = \left[\begin{array}{c} 
X \\ D \end{array} \right]
$$
has linearly independent columns.  In the case $p \leq n$, the columns of $X$
will be linearly independent almost surely (the argument for this is similar to
the arguments in the proof of Lemma \ref{lem:x_cont_gp}), so the columns of 
$M$ will be linearly independent almost surely.

Thus assume $p>n$. Let $q=\nuli(D)$, so $r=\rank(D)=p-q$.  Pick $r$ columns of
$D$ that are linearly independent; then the corresponding columns of $M$ are
linearly independent.  It now suffices to check linear independence of the
remaining $p-r$ columns of $M$.  But any $n$ columns of $X$ will be linearly
independent almost surely (again, the argument for this is similar to the
arguments from the proof of Lemma \ref{lem:x_cont_gp}), so the result is given
provided $p-r \leq n$, \ie, $q \leq n$.  
\hfill\qedsymbol

\subsection{Proof of Corollary \ref{cor:uniqueness_centered}}

Let $V = [\,V_1 \; V_{-1} \,] \in \R^{n \times n}$ be an orthogonal matrix, 
where \smash{$V_1 = \ones/\sqrt{n} \in \R^{n \times 1}$} and $V_{-1} \in \R^{n
  \times (n-1)}$ has columns that span $\col(M)$.  Note that the centered
generalized lasso criterion in \eqref{eq:genlasso_centered} can be written as  
$$
 \frac{1}{2} \| y - M X \beta \|_2^2 + \lambda \| D \beta \|_1 =
 \frac{1}{2} \| V_1^T y \|_2^2 + \|V_{-1}^T y - V_{-1}^T X \beta\|_2^2 + 
\lambda \| D \beta \|_1,
$$
hence problem \eqref{eq:genlasso_centered} is equivalent to a regular
(uncentered) generalized lasso problem with response \smash{$V_{-1}^T y \in
  \R^{n-1}$} and predictor matrix \smash{$V_{-1}^T X \in \R^{(n-1) \times p}$}.
By straightforward arguments (using integration and change of variables), 
$(X,y)$ having a density on $\R^{np+n}$ implies that \smash{$(V_{-1}^T
  X,V_{-1}^T y)$} has a density on $\R^{(n-1)p+(n-1)}$. Thus, we can
apply Theorem \ref{thm:uniqueness} to the generalized lasso problem with
response \smash{$V_{-1}^T y$} and predictor matrix \smash{$V_{-1}^T X$} to give
the desired result.  
\hfill\qedsymbol

\subsection{Proof of Lemma \ref{lem:z_scaled}}

Let $\sigma^{n-1}$ denote the $(n-1)$-dimensional spherical measure, which is
just a normalized version of the $(n-1)$-dimensional Hausdorff measure
$\cH^{n-1}$ on the unit sphere $\S^{n-1}$, i.e., defined by
\begin{equation}
\label{eq:sm_haus}
\sigma^{n-1}(S) = \frac{\cH^{n-1}(S)}{\cH^{n-1}(\S^{n-1})}, \quad \text{for $S
  \subseteq \S^{n-1}$}.
\end{equation}
Thus, it is sufficient to prove that the distribution of $Z/\|Z\|_2$ is
absolutely continuous with respect to $\sigma^{n-1}$.  For this, it is helpful
to recall that an alternative definition of the $(n-1)$-dimensional spherical
measure, for an arbitrary $\alpha>0$, is
\begin{equation}
\label{eq:sm_leb}
\sigma^{n-1}(S) = \frac{\cL^n(\cone_\alpha(S))}{\cL(\B^n_\alpha)}, \quad 
 \text{for $S\subseteq \S^{n-1}$}.
\end{equation}
where $\cL^n$ denotes $n$-dimensional Lebesgue measure,
$\B^n_\alpha=\{ x \in \R^n : \|x\|_2 \leq \alpha \}$ is the $n$-dimensional ball  
of radius $\alpha$, and $\cone_\alpha(S)= \{ tx : x \in S, \; t \in [0,\alpha]
\}$.  That \eqref{eq:sm_leb} and \eqref{eq:sm_haus} coincide is due to the fact
that any two measures that are uniformly distributed over a separable metric
space must be equal up to a positive constant (see Theorem 3.4 in 
\citet{mattila1995geometry}), and as both \eqref{eq:sm_leb} and
\eqref{eq:sm_haus} are probability measures on $\S^{n-1}$, this positive
constant must be 1.  

Now let $S \subseteq \S^{n-1}$ be a set of null spherical measure,
$\sigma^{n-1}(S)=0$.  From the representation for spherical measure in
\eqref{eq:sm_leb}, we see that \smash{$\cL^n(\cone_\alpha(S))=0$} for
any $\alpha>0$.  Denoting $\cone(S)=\{tx : x \in S, \; t \geq 0\}$, we have
$$
\cL^n(\cone(S)) = \cL^n\bigg(\bigcup_{k=1}^\infty \cone_k(S)\bigg) 
\leq \sum_{k=1}^\infty \cL^n(\cone_k(S)) = 0.
$$
This means that $\P(Z \in \cone(S))=0$, as the distribution of $Z$ is
absolutely continuous with respect to $\cL^n$, and moreover $\P(Z/\|Z\|_2 \in 
S)=0$, since $Z \in \cone(S) \iff Z \in Z/\|Z\|_2\in S$.  This completes the
proof. 
\hfill\qedsymbol

\subsection{Proof of Lemma \ref{lem:s_cap_a}}

Denote the $n$-dimensional unit ball by $\B^n=\{x \in \R^n: \|x\|_2 \leq 1\}$.
Note that the relative boundary of $\B^n \cap A$ is precisely
$$
\relbd(\B^n \cap A) = \S^{n-1} \cap A.
$$
The boundary of a convex set has Lebesgue measure zero (see Theorem 1 in
\citet{lang1986note}), and so we claim $\S^{n-1} \cap A$ has
$(n-1)$-dimensional Hausdorff measure zero.  To see this, note first that we can
assume without a loss of generality that $\dim(A)=n-1$, else the claim follows
immediately.  We can now interpret $\B^n \cap A$ as a set in the ambient space
$A$, which is diffeomorphic---via a change of basis---to $\R^{n-1}$.  To be more 
precise, if $V \in \R^{n \times (n-1)}$ is a matrix whose columns are
orthonormal and span the linear part of $A$, and $a \in A$ is arbitrary, then
$V^T(\B^n \cap A - a) \subseteq \R^{n-1}$ is a convex set, and by the fact cited
above its boundary must have $(n-1)$-dimensional Lebesgue measure zero.  It can
be directly checked that 
$$
\bd(V^T(\B^n \cap A - a)) = V^T (\relbd(\B^n \cap A) - a) = V^T 
(\S^{n-1} \cap A - a).
$$
As the $(n-1)$-dimensional Lebesgue measure and $(n-1)$-dimensional Hausdorff 
measure coincide on $\R^{n-1}$, we see that $V^T (\S^{n-1} \cap A - a)$ has
$(n-1)$-dimensional Hausdorff measure zero.  Lifting this set back to $\R^n$,
via the transformation  
$$
V V^T (\S^{n-1} \cap A - a) + a = \S^{n-1} \cap A,
$$
we see that $\S^{n-1} \cap A$ too must have Hausdorff measure zero, the desired
result, because the map $x \mapsto V x + a$ is Lipschitz (then apply, \eg,
Theorem 1 in Section 2.4.1 of \citet{evans1992measure}).  
\hfill\qedsymbol

\subsection{Proof of Lemma \ref{lem:x_scaled}}

Let us abbreviate \smash{$\tilde{X}=XW_{X}^{-1}$} for the scaled predictor
matrix, whose columns are \smash{$\tilde{X}_i=X_i/\|X_i\|_2$},
$i=1,\ldots,p$. By similar arguments to those given in the proof of Lemma
\ref{lem:x_cont_gp}, to show \smash{$\tilde{X}$} is in $D$-GP almost
surely, it suffices to show that for each $i=1,\ldots,p$,  
$$
\P\big( \tilde{X}_i \in A \,\big|\, \tilde{X}_j, j \neq i \big) = 0,  
$$
where $A \subseteq \R^n$ is an affine space depending on
\smash{$\tilde{X}_j$}, $j \neq i$.  This follows by applying our previous two
lemmas: the distribution of \smash{$\tilde{X}_i$} is absolutely continuous with
respect $(n-1)$-dimensional Hausdorff measure on $\S^{n-1}$, by Lemma
\ref{lem:z_scaled}, and $\S^{n-1} \cap A$ has $(n-1)$-dimensional Hausdorff
measure zero, by Lemma \ref{lem:s_cap_a}. 

To establish that the null space condition \smash{$\nul(\tilde{X}) \cap \nul(D)
  = \{0\}$} holds almost surely, note that the proof of Lemma
\ref{lem:x_cont_null} really only depends on the fact that any collection 
of $k$ columns of $X$, for $k \leq n$, are linearly independent almost
surely.  It can be directly checked that the scaled columns of
\smash{$\tilde{X}$} share this same property, and thus we can repeat the
same arguments as in Lemma \ref{lem:x_cont_null} to give the result.
\hfill\qedsymbol

\subsection{Proof of Corollary \ref{cor:uniqueness_standard}}

Let $V = [\,V_1 \; V_{-1} \,] \in \R^{n \times n}$ be as in the proof of
Corollary \ref{cor:uniqueness_centered}, and rewrite the criterion in 
\eqref{eq:genlasso_standard} as  
$$
 \frac{1}{2} \| y - M X W_{MX}^{-1} \beta \|_2^2 + \lambda \| D \beta \|_1 = 
 \frac{1}{2} \| V_1^T y \|_2^2 + \|V_{-1}^T y - V_{-1}^T X W_{MX}^{-1}
 \beta\|_2^2 + \lambda \| D \beta \|_1.
$$
Now for each $i=1,\ldots,p$, note that \smash{$\|V_{-1}^T X_i\|_2^2=
X_i^T V_{-1} V_{-1}^T X_i = \|MX_i\|_2^2$}, which means that 
$$
V_{-1}^T X W_{MX} = V_{-1}^T X  W_{V_{-1}^T  X}^{-1},
$$
precisely the scaled version of \smash{$V_{-1}^T X$}. From the
second to last display, we see that the standardized generalized lasso problem
\eqref{eq:genlasso_standard} is the same as a scaled generalized lasso problem
with response \smash{$V_{-1}^T y$} and scaled predictor matrix \smash{$V_{-1}^T
  X W_{V_{-1}^T  X}^{-1}$}.   Under the conditions placed on $y,X$, as
explained in the proof of Corollary \ref{cor:uniqueness_centered}, 
the distribution of \smash{$(V_{-1}^T X,V_{-1}^T y)$} is absolutely
continuous. Therefore we can apply Corollary \ref{cor:uniqueness_scaled} to give 
the result. 
\hfill\qedsymbol

\subsection{Proof of Lemma \ref{lem:dual_g}}

Write $h(\beta)=\lambda \|D\beta\|_1$. We may rewrite problem
\eqref{eq:genlasso_g} as thus
\begin{equation}
\label{eq:genlasso_g0}
\minimize_{\beta \in \R^p, \, z \in \R^n} \; G(z) + h(\beta) \quad\st\quad
z=X\beta.
\end{equation}
The Lagrangian of the above problem is
\begin{equation}
\label{eq:lag}
L(\beta,z,v) =  G(z) + h(\beta) + v^T (z - X\beta),
\end{equation}
and minimizing the Lagrangian over $\beta,z$ gives the dual problem 
\begin{equation}
\label{eq:dual_g0}
\maximize_{v \in \R^n} \; -G^*(-v) - h^*(X^T v),
\end{equation}
where $G^*$ is the conjugate of $G$, and $h^*$ is the conjugate of $h$.  Noting 
that \smash{$h(\beta)=\max_{\eta \in D^T B^m_\infty(\lambda)} \eta^T
  \beta $}, we have 
$$
h^*(\alpha) = I_{D^T B^m_\infty(\lambda)}(\alpha) 
= \begin{cases}
0 & \alpha \in D^T B^m_\infty(\lambda) \\
\infty & \text{otherwise}
\end{cases},
$$
and hence the dual problem \eqref{eq:dual_g0} is equivalent to the claimed
one \eqref{eq:dual_g}.

As $G$ is essentially smooth and essentially strictly convex, the
interior of its domain is nonempty.  Since the domain of  
$h$ is all of $\R^p$, this is enough to ensure that strong duality holds between 
\eqref{eq:genlasso_g0} and \eqref{eq:dual_g0} (see, \eg, Theorem 28.2 of
\citet{rockafellar1970convex}). Moreover, if a solution 
\smash{$\hbeta,\hat{z}$} is attained in \eqref{eq:genlasso_g0}, and a solution
\smash{$\hv$} is attained in \eqref{eq:dual_g0}, then by minimizing the
Lagrangian \smash{$L(\beta,z,\hv)$} in \eqref{eq:lag} over $z$ and $\beta$, we 
have the relationships   
\begin{equation}
\label{eq:primal_dual_g0}
\nabla G(\hat{z}) = -\hv, \quad\text{and}\quad
X^T \hv \in \partial h(\hbeta),
\end{equation}
respectively, where $\partial h(\cdot)$ is the subdifferential operator of $h$.
The first relationship in \eqref{eq:primal_dual_g0} can be rewritten as
\smash{$\nabla G(X\hbeta) = -\hv$}, matching the first relationship in
\eqref{eq:primal_dual_g}. The second relationship in \eqref{eq:primal_dual_g0}
can be rewritten as \smash{$D^T \hu \in \partial h(\hbeta)$}, where \smash{$\hu
  \in B^m_\infty(\lambda)$} is such that \smash{$X^T \hv = D^T \hu$}, and thus 
we can see that \smash{$\hu/\lambda$} is simply a relabeling of the subgradient
\smash{$\hgamma$} of the $\ell_1$ norm evaluated at \smash{$D\hbeta$}, matching
the second relationship in \eqref{eq:primal_dual_g}. 

Finally, we address the constraint qualification conditions \eqref{eq:dom_g},
\eqref{eq:ran_g}.  When \eqref{eq:dom_g} holds, we know that $G^*$ has no
directions of recession, and so if $C \neq \emptyset$, then the dual problem
\eqref{eq:dual_g} has a solution (see, \eg, Theorems 27.1 and 27.3 in
\citet{rockafellar1970convex}), equivalently, problem \eqref{eq:dual_g0} has a
solution. Suppose \eqref{eq:ran_g} also holds, or equivalently, 
$$
(-C) \cap \inte(\dom(G^*)) \neq \emptyset,
$$
which follows as $\inte(\dom(G^*)) = \inte(\ran(\nabla G))$, due to the fact 
that the map $\nabla G : \inte(\dom(G)) \to \inte(\dom(G^*))$ is a
homeomorphism.  Then we have know further that \smash{$-\hv \in
  \inte(\dom(G^*))$} by essential smoothness and essential strict convexity of 
$G^*$ (in particular, by the property that $\|\nabla G^*\|_2$ diverges
along any sequence convering to a boundary point of $\dom(G^*)$;
see, \eg, Theorem 3.12 in \citet{bauschke1997legendre}), so 
\smash{$\hat{z}  = \nabla G^*(-\hv)$} is well-defined; by 
construction it satisfies the first relationship in \eqref{eq:primal_dual_g0},
and minimizes the Lagrangian \smash{$L(\beta,z,\hv)$} over $z$.  The
second relationship in \eqref{eq:primal_dual_g0}, recall, can be rewritten as  
\smash{$D^T \hu \in \partial h(\hbeta)$}; that the Lagrangian
\smash{$L(\beta,z,\hv)$} attains its infimum over $\beta$ follows from 
the fact that the map \smash{$\beta \mapsto h(\beta)-\hu^T D\beta$} has no  
strict directions of recession (directions of recession in which this map is not 
constant). We have shown that the Lagrangian \smash{$L(\beta,z,\hv)$} attains 
its infimum over $\beta,z$.  By strong duality, this is enough to ensure that
problem \eqref{eq:genlasso_g0} has a solution, equivalently, that problem
\eqref{eq:genlasso_g} has a solution, completing the proof. 

\subsection{Proof of Lemma \ref{lem:exist_glm_unreg}}

When $\lambda=0$, note that $C=\nul(X^T)$, so \eqref{eq:ran_psi} becomes
\eqref{eq:ran_psi_unreg}.  For Poisson regression, the condition
\eqref{eq:ran_psi_pois} is an immediate rewriting of \eqref{eq:ran_psi_unreg},
because $\inte(\ran(\nabla \psi)) = \R^n_{++}$, where $\R_{++}=(0,\infty)$
denotes the positive real numbers.  For logistic regression, the argument
leading to \eqref{eq:ran_psi_log} is a little more tricky, and is given below.

Observe that in the logistic case, $\inte(\ran(\nabla \psi)) = (0,1)^n$, hence 
condition \eqref{eq:ran_psi_unreg} holds if and only if there exists $a \in
(0,1)^n$ such that $X^T (y-a)=0$, i.e., there exists $a' \in (0,1)^n$ such that
$X^T D_Y a' = 0$, where $D_Y=\diag(Y_1,\ldots,Y_n)$.  The latter statement is
equivalent to   
\begin{equation}
\label{eq:ran_psi_log1}
\nul(X^T D_Y) \cap \R^n_{++} \neq \emptyset.
\end{equation}
We claim that this is actually in turn equivalent to  
\begin{equation}
\label{eq:ran_psi_log2}
\col(D_Y X) \cap \R^n_+ = \{0\}.
\end{equation}
where $\R_+=[0,\infty)$ denotes the nonnegative real numbers, which would
complete the proof, as the claimed condition \eqref{eq:ran_psi_log2} is a direct
rewriting of \eqref{eq:ran_psi_log}.

Intuitively, to see the equivalence of \eqref{eq:ran_psi_log1} and
\eqref{eq:ran_psi_log2}, it helps to draw a picture: the two subspaces $\col(D_Y
X)$ and $\nul(X^T D_Y)$ are orthocomplements, and if the former only intersects
the nonnegative orthant at $0$, then the latter must pass through the negative 
orthant. This intuition is formalized by Stiemke's lemma.  This is a theorem of
alternatives, and a close relative of Farkas' lemma (see, \eg, Theorem 2 in
Chapter 1 of \citet{kemp1978introduction}); we state it below for reference. 

\begin{lemma}
Given $A \in \R^{n \times p}$, exactly one of the following systems has a
solution: 
\begin{itemize}
\item $Ax=0$, $x<0$ for some $x \in \R^p$;
\item $A^T y \geq 0$ for some $y \in\R^n$, $y \neq 0$.
\end{itemize}
\end{lemma}

\noindent
Applying this lemma to $A=X^T D_Y$ gives the equivalence of
\eqref{eq:ran_psi_log1} and \eqref{eq:ran_psi_log2}, as desired.
\hfill\qedsymbol

\subsection{Proof of Lemma \ref{lem:exist_glm_reg}}

We prove the result for the logistic case; the result for the Poisson case
follows similarly.  Recall that in the logistic case, $\inte(\ran(\nabla\psi)) = 
(0,1)^n$.  Given $y \in \{0,1\}^n$, and arbitrarily small $\epsilon>0$, note
that we can always write $y=z+\delta$, where $z \in (0,1)^n$ and \smash{$\delta
  \in B^m_\infty(\epsilon)$}. Thus \eqref{eq:ran_psi} holds as long as 
$$
C = (X^T)^{-1} \big(D^T B^m_\infty(\lambda)\big)
= \big\{ u \in \R^n : X^T u = D^T v, \; v \in B^m_\infty(\lambda) \big\} 
$$
contains a $\ell_\infty$ ball of arbitrarily small radius centered at the 
origin.  As $\lambda > 0$, this holds provided $\row(X) \subseteq
\row(D)$, i.e., $\nul(D) \subseteq \nul(X)$, as claimed.
\hfill\qedsymbol

\subsection{Proof of Lemma \ref{lem:implicit_glm}}

We first establish \eqref{eq:fit_glm}, \eqref{eq:sol_glm}.  
Multiplying both sides of stationarity condition \eqref{eq:stat_glm} by
\smash{$P_{\nul(D_{-\cB})}$} yields
$$
P_{\nul(D_{-\cB})} X^T \big(y-\nabla \psi(X \hbeta)\big) = \lambda
P_{\nul(D_{-\cB})} D_\cB^T s.
$$
Let us abbreviate \smash{$M=P_{\nul(D_{-\cB})} X^T$}.  After rearranging, the
above becomes 
$$
M\nabla \psi(X \hbeta) = M (y - \lambda M^+ P_{\nul(D_{-\cB})} D_\cB^T s).    
$$
where we have used 
\smash{$P_{\nul(D_{-\cB})} D_\cB^T s = MM^+ P_{\nul(D_{-\cB})} D_\cB^T s$}, 
which holds as \smash{$P_{\nul(D_{-\cB})} D_\cB^T s \in
  \col(M)$}, from the second to last display. Moreover, we can simplify the
above, using \smash{$M^+ P_{\nul(D_{-\cB})} = M^+$}, to yield
$$
M \nabla \psi(X \hbeta) = M (y - \lambda M^+ D_\cB^T s), 
$$
and multiplying both sides by $M^+$,
\begin{equation}
\label{eq:stat_glm_b1}
P_{\row(M)} \nabla \psi(X \hbeta) = P_{\row(M)} (y - \lambda M^+ D_\cB^T s). 
\end{equation}
Lastly, by virtue of the fact that \smash{$D_{-\cB} \hbeta = 0$}, we have 
\smash{$X\hbeta = X P_{\nul(D_{-\cB})} \hbeta = M^T \hbeta \in \row(M)$}, so 
\begin{equation}
\label{eq:stat_glm_b2}
P_{\nul(M)} X \hbeta = 0.
\end{equation}

We will now show that \eqref{eq:stat_glm_b1}, \eqref{eq:stat_glm_b2} together
imply \smash{$\nabla \psi(X \hbeta)$} can be expressed in terms of a certain 
Bregman projection onto an affine subspace, with respect to $\psi^*$.  To this
end, consider 
$$
\hat{x} = P_S^f(a) = \argmin_{x \in S} \Big(f(x)-f(a)-
\langle \nabla f(a), x-a \rangle\Big),
$$
for a function $f$, point $a$, and set $S$.  The first-order optimality
conditions are 
$$
\big\langle \nabla f(\hat{x}) - \nabla f(a), 
z - \hat{x} \big\rangle \geq 0 \;\,\text{for all $z \in S$}, 
\quad\text{and}\quad \hat{x} \in S.
$$
When $S$ is an affine subspace, \ie, $S=c+L$ for a point $c$ and linear subspace
$L$, this reduces to 
$$
\big\langle \nabla f(\hat{x}) - \nabla f(a), 
v \big\rangle = 0 \;\,\text{for all $v \in L$},
\quad\text{and}\quad \hat{x} \in c+L.
$$
\ie,
\begin{equation}
\label{eq:breg_opt}
P_L \nabla f(\hat{x}) = P_L \nabla f(a),
\quad\text{and}\quad P_{L^\perp} \hat{x} = P_{L^\perp} c. 
\end{equation}
In other words, \smash{$\hat{x}=P_S^f(a)$}, for $S=c+L$, if and only if
\eqref{eq:breg_opt} holds.

Set \smash{$\hat{x}=\nabla \psi(X\hbeta)$}, $f=\psi^*$, $a=\nabla \psi(0)$, 
\smash{$c=y-\lambda M^+D_{-\cB}^T s$}, and $L=\nul(M)$. We see that
\eqref{eq:stat_glm_b1} is equivalent to 
\smash{$P_{L^\perp} \hat{x} = P_{L^\perp} c$}.  Meanwhile, using \smash{$(\nabla
  \psi)^{-1} = \nabla \psi^*$} as guaranteed by essential smoothness and
essential strict convexity of $\psi$, we see that \eqref{eq:stat_glm_b2} is
equivalent to \smash{$P_{\nul(M)} \nabla \psi^* (\nabla\psi( X\hbeta)) = 0$},
in turn equivalent to \smash{$P_L \nabla f(\hat{x}) = P_L \nabla f(a)$}.  From
the first-order optimality conditions \eqref{eq:breg_opt}, this shows that
\smash{$\nabla \psi (X\hbeta) = P_{c+L}^f(a) = 
P_{y-K_{\cB,s}}^{\psi^*}(\nabla \psi(0))$}.  Using \smash{$(\nabla
  \psi)^{-1} = \nabla \psi^*$}, once again, establishes \eqref{eq:fit_glm}.

As for \eqref{eq:sol_glm}, this follows by simply writing \eqref{eq:fit_glm} as 
$$
M^T \hbeta = \nabla \psi^* \Big( P^{\psi^*}_{y-K_{\cB,s}}\big(\nabla \psi(0)    
\big)\Big),
$$
where we have again used \smash{$X\hbeta = X P_{\nul(D_{-\cB})} \hbeta = M^T
  \hbeta$}.  Solving the above linear system for \smash{$\hbeta$} gives
\eqref{eq:sol_glm}, where \smash{$b \in \nul(M^T)=\nul(XP_{\nul(D_{-\cB})})$}.
This constraint together with \smash{$b \in \nul(D_{-\cB})$} implies
\smash{$b \in \nul(X) \cap \nul(D_{-\cB})$}, as claimed.

Finally, the results with $\cA,r$ in place of $\cB,s$ follow similarly.  We
begin by multiplying both sides of \eqref{eq:stat_glm} by
\smash{$P_{\nul(D_{-\cA})}$}, and then proceed with the same chain of arguments 
as above. 
\hfill\qedsymbol

\subsection{Proof of Lemma \ref{lem:local_glm}}

The proof follows a similar general strategy to that of Lemma 9 in
\citet{tibshirani2012degrees}. We will abbreviate $\cB=\cB(y)$, $s=s(y)$,
$\cA=\cA(y)$, and $r=r(y)$. Consider the representation for
\smash{$\hbeta(y)$} in \eqref{eq:sol_glm} of Lemma \ref{lem:implicit_glm}.  As
the active set is $\cA$, we know that
$$
D_{\cB \setminus \cA} (XP_{\nul(D_{-\cB})})^+ 
\nabla \psi^* \Big( P^{\psi^*}_{y-K_{\cB,s}}\big(\nabla \psi(0) \big)\Big) + 
D_{\cB \setminus \cA} b = 0,
$$
\ie, 
$$
D_{\cB \setminus \cA} (XP_{\nul(D_{-\cB})})^+  
\nabla \psi^* \Big( P^{\psi^*}_{y-K_{\cB,s}}\big(\nabla \psi(0) \big)\Big) =  
-D_{\cB \setminus \cA} b \in D_{\cB \setminus \cA} \big(\nul(X) \cap 
\nul(D_{-\cB})\big),
$$
and so
$$
P_{[D_{\cB \setminus \cA}(\nul(X) \cap \nul(D_{-\cB}))]^\perp}  
D_{\cB \setminus \cA} (XP_{\nul(D_{-\cB})})^+ 
\nabla \psi^* \Big( P^{\psi^*}_{y-K_{\cB,s}}\big(\nabla \psi(0) \big)\Big) = 0.
$$
Recalling $M_{\cA,\cB}$ as defined in \eqref{eq:m}, and abbreviating 
\smash{$\hat{x}=P^{\psi^*}_{y-K_{\cB,s}}(\nabla \psi(0))$}, we may write this
simply as
$$
\nabla \psi^* (\hat{x}) \in \nul(M_{\cA,\cB}).
$$
Since \smash{$\nabla\psi^*(\hat{x}) = X\hbeta(y)$}, we have
\smash{$\nabla\psi^*(\hat{x}) \in \col(XP_{\nul(D_{-\cB})})$}, so
combining this with above display, and using \smash{$(\nabla\psi^*)^{-1} =
  \nabla\psi$}, gives 
$$
\hat{x} \in \nabla\psi\big( \col(XP_{\nul(D_{-\cB})}) \cap
\nul(M_{\cA,\cB})\big).  
$$
And since \smash{$\hat{x} \in y-K_{\cB,s}$}, with \smash{$K_{\cB,s}$} an
affine space, as defined in \eqref{eq:k}, we have \smash{$y \in \hat{x} + 
  K_{\cB,s}$}, which combined with the last display implies
$$
y \in K_{\cB,s} + \nabla\psi\big( \col(XP_{\nul(D_{-\cB})}) \cap
\nul(M_{\cA,\cB})\big). 
$$
But as $y \notin \cN$, where the set $\cN$ is defined in \eqref{eq:n}, we arrive
at  
$$
M_{\cA,\cB} = P_{[D_{\cB \setminus \cA}(\nul(X) \cap \nul(D_{-\cB}))]^\perp}    
D_{\cB \setminus \cA} (XP_{\nul(D_{-\cB})})^+ = 0,
$$
which means 
\begin{equation}
\label{eq:fact_ab}
\col\big(D_{\cB \setminus \cA} (XP_{\nul(D_{-\cB})})^+\big) \subseteq
D_{\cB \setminus \cA}\big(\nul(X) \cap \nul(D_{-\cB})\big).
\end{equation}
This is an important realization that we will return to shortly.

As for the optimal subgradient \smash{$\hgamma(y)$} corresponding to 
\smash{$\hbeta(y)$}, note that we can write
\begin{align}
\nonumber
\hgamma_\cB(y) &= \lambda s, \\
\label{eq:gamma_y}
\hgamma_{-\cB}(y) &= \frac{1}{\lambda} (D_{-\cB}^T)^+\Big[
X^T \Big(y - P^{\psi^*}_{y-K_{\cB,s}}\big(\nabla \psi(0) \big)\Big) - 
\lambda D_\cB^T s \Big]+c,
\end{align}
for some \smash{$c \in \nul(D_{-\cB}^T)$}.  The first expression holds by
definition of $\cB,s$, and the second is a result of solving for
\smash{$\hgamma_{-\cB}(y)$} in the stationarity condition \eqref{eq:stat_glm},
after plugging in for the form of the fit in \eqref{eq:fit_glm}.  

Now, at a new response $y'$, consider defining
\begin{align*}
\hbeta(y') &= (XP_{\nul(D_{-\cB})})^+ \nabla \psi^* 
\Big( P^{\psi^*}_{y'-K_{\cB,s}}\big(\nabla \psi(0) \big)\Big) + b', \\   
\hgamma_\cB(y') &= \lambda s, \\
\hgamma_{-\cB}(y') &= \frac{1}{\lambda} (D_{-\cB}^T)^+\Big[
X^T \Big(y' - P^{\psi^*}_{y'-K_{\cB,s}}\big(\nabla \psi(0) \big)\Big) - 
\lambda D_\cB^T s \Big]+c,
\end{align*}
for some \smash{$b' \in \nul(X) \cap \nul(D_{-\cB})$} to be specified later, and
for the same value of \smash{$c \in \nul(D_{-\cB}^T)$} as in
\eqref{eq:gamma_y}.  By the same arguments as given at the end of the proof of
Lemma \ref{lem:dual_g}, where we discussed the constraint qualification
conditions \eqref{eq:dom_g}, \eqref{eq:ran_g}, the Bregman projection
\smash{$P^{\psi^*}_{y'-K_{\cB,s}}(\nabla \psi(0))$} in the above expressions
is well-defined, for any $y'$, under \eqref{eq:dom_psi}.  However, this Bregman
projection need not lie in $\inte(\dom(\psi^*))$---and therefore 
\smash{$\nabla\psi^*(P^{\psi^*}_{y'-K_{\cB,s}}(\nabla \psi(0)))$} need not be
well-defined---unless we have the additional condition $y' \in
\inte(\ran(\nabla\psi)) + C$.  Fortunately, under \eqref{eq:ran_psi}, the latter 
condition on $y'$ is implied as long as $y'$ is sufficiently close to $y$, \ie,
there exists a neighborhood $U_0$ of $y$ such that $y' \in
\inte(\ran(\nabla\psi)) + C$, provided $y' \in U_0$.  By Lemma \ref{lem:dual_g}, 
we see that a solution in \eqref{eq:genlasso_g} exists at such a point
$y'$.  In what remains, we will show that this solution and its optimal
subgradient obey the form in the above display.     

Note that, by construction, the pair \smash{$(\hbeta(y'),\hgamma(y'))$} defined
above satisfy the stationarity condition \eqref{eq:stat_glm} at $y'$, and
\smash{$\hgamma(y')$} has boundary set and boundary signs $\cB,s$.  It remains   
to show that \smash{$(\hbeta(y'),\hgamma(y'))$} satisfy the subgradient
condition \eqref{eq:subg_g}, and that \smash{$\hbeta(y')$} has active set and
active signs $\cA,r$; equivalently, it remains to verify the following three
properties, for $y'$ sufficiently close to $y$, and for an appropriate choice of
$b'$:   
\begin{enumerate}[(i)]
\item $\|\hgamma_{-\cB}(y')\|_\infty < 1$;
\item $\supp(D\hbeta(y'))=\cA$;
\item $\sign(D_\cA\hbeta(y'))=r$.
\end{enumerate}

Because \smash{$\hgamma(y)$} is a subgradient corresponding to
\smash{$\hbeta(y)$}, and has boundary set and boundary signs $\cB,s$, we know
that \smash{$\hgamma_{-\cB}(y)$} in \eqref{eq:gamma_y} has $\ell_\infty$ norm
strictly less than 1.  Thus, by continuity of 
$$
x \mapsto \bigg\|\frac{1}{\lambda} (D_{-\cB}^T)^+\Big[
X^T \Big(x - P^{\psi^*}_{x-K_{\cB,s}}\big(\nabla \psi(0) \big)\Big) - 
\lambda D_\cB^T s \Big]+c\, \bigg\|_\infty
$$
at $y$, which is implied by continuity of \smash{$x \mapsto
  P^{\psi^*}_{x-K_{\cB,s}}(\nabla \psi(0))$} at $y$, by Lemma
\ref{lem:breg_cont}, we know that there exists some neighborhood $U_1$ of $y$
such that property (i) holds, provided $y' \in U_1$.

By the important fact established in \eqref{eq:fact_ab}, we see that there
exists \smash{$b' \in \nul(X) \cap \nul(D_{-\cB})$} such that 
$$
D_{\cB \setminus \cA} b' = -D_{\cB \setminus \cA} (XP_{\nul(D_{-\cB})})^+  
\nabla \psi^* \Big( P^{\psi^*}_{y'-K_{\cB,s}}\big(\nabla \psi(0) \big)\Big),
$$
which implies that \smash{$D_{\cB \setminus \cA} \hbeta(y')=0$}.  To verify
properties (ii) and (iii), we must show this choice of $b'$ is such that 
\smash{$D_\cA \hbeta(y')$} is nonzero in every coordinate and has signs
matching $r$.  Define a map
$$
T(x) = (XP_{\nul(D_{-\cB})})^+  
\nabla \psi^* \Big( P^{\psi^*}_{x-K_{\cB,s}}\big(\nabla \psi(0) \big)\Big),
$$
which is continuous at $y$, again by continuity of \smash{$x \mapsto 
  P^{\psi^*}_{x-K_{\cB,s}}(\nabla \psi(0))$} at $y$, by Lemma
\ref{lem:breg_cont}. Observe that
$$
D_\cA \hbeta(y') = D_\cA T(y') + D_\cA b' = 
D_\cA T(y') + D_\cA b + D_\cA (b-b'). 
$$
As \smash{$D_\cA\hbeta(y) = D_\cA
  T(y)+D_\cA b$} is nonzero in every coordinate and has signs equal to $r$, by 
definition of $\cA,r$, and $T$ is continuous at $y$, there exists a 
neighborhood $U_2$ of $y$ such that \smash{$D_\cA T(y')+D_\cA b$} is nonzero 
in each coordinate with signs matching $r$, provided $y' \in U_2$.  Furthermore, as  
$$
 \|D_\cA(b-b')\|_\infty \leq \|D^T\|_{2,\infty} \|b-b'\|_2,  
$$
where \smash{$\|D^T\|_{2,\infty}$} denotes the maximum $\ell_2$ norm of rows of 
$D$, we see that \smash{$D_\cA T(y')+D_\cA b'$} will be nonzero in each  
coordinate with the correct signs, provided $b'$ can be chosen arbitrarily close
to $b$, subject to the restrictions \smash{$b' \in \nul(X) \cap \nul(D_{-\cB})$} 
and \smash{$D_{\cB \setminus \cA} b' = -D_{\cB \setminus \cA} T(y')$}.   

Such a $b'$ does indeed exist, by the bounded inverse theorem.  Let
\smash{$L=\nul(X) \cap \nul(D_{-\cB})$}, and 
\smash{$N=\nul(D_{\cB \setminus \cA}) \cap L$}.  Consider the linear map 
\smash{$D_{\cB \setminus \cA}$}, viewed as a function from $L/N$ 
(the quotient of $L$ by $N$) to \smash{$D_{\cB \setminus \cA}(L)$}: this is a  
bijection, and therefore it has a bounded 
inverse.  This means that there exists some $R>0$ such that  
$$
\|b-b'\|_2 \leq R \big\|D_{\cB \setminus \cA} T(y') -  
D_{\cB \setminus \cA} T(y) \big\|_2, 
$$
for a choice of \smash{$b' \in \nul(X) \cap \nul(D_{-\cB})$}
with \smash{$D_{\cB \setminus \cA} b' = -D_{\cB \setminus \cA} T(y')$}. 
By continuity of $T$ at $y$, once again, there exists a neighborhood $U_3$ of 
$y$ such that the right-hand side above is sufficiently small, \ie, such that 
$\|b-b'\|_2$ is sufficiently small, provided $y' \in U_3$. 

Finally, letting $U=U_0 \cap U_1 \cap U_2 \cap U_3$, we see that we have
established properties (i), (ii), and (iii), and hence the desired result,
provided $y' \in U$.  
\hfill\qedsymbol

\subsection{Continuity result for Bregman projections}

\begin{lemma}
\label{lem:breg_cont}
Let $f,f^*$ be a conjugate pair of Legendre (essentially smooth and essentially
strictly convex) functions on $\R^n$, with $0 \in \inte(\dom(f^*))$.  Let $S 
\subseteq \R^n$ be a nonempty closed convex set.  Then the Bregman projection 
map       
$$
x \mapsto P^f_{x-S}\big(\nabla f^*(0)\big)
$$
is continuous on all of $\R^n$.  Moreover, \smash{$P^f_{x-S}(\nabla f^*(0))
  \in \inte(\dom(f))$} for any $x \in \inte(\dom(f))+S$.  
\end{lemma}

\begin{proof}
As $0 \in \inte(\dom(f^*))$, we know that $f$ has no directions of recession   
(\eg, by Theorems 27.1 and 27.3 in \citet{rockafellar1970convex}), thus
the Bregman projection \smash{$P^f_{x-S}(\nabla f^*(0))$} is well-defined for
any $x \in \R^n$.  Further, for $x-S \in \inte(\dom(f))$, we know that  
\smash{$P^f_{x-S}(\nabla f^*(0)) \in \inte(\dom(f))$}, by essential smoothness
of $f$ (by the property that $\|\nabla f\|_2$ approaches $\infty$
along any sequence that converges to boundary point of $\dom(f)$; \eg, 
see Theorem 3.12 in \citet{bauschke1997legendre}).

It remains to verify continuity of \smash{$x \mapsto P^f_{x-S}(\nabla
  f^*(0))$}.  Write \smash{$P^f_{x-S}(\nabla   f^*(0))=\hv$},
where \smash{$\hv$} is the unique solution of 
$$
\minimize_{v \in x-S} \; f(v),
$$
or equivalently, \smash{$P^f_{x-S}(\nabla f^*(0))=\hat{w}+x$}, where
\smash{$\hat{w}$} is the unique solution of 
$$
\minimize_{w \in -S} \; f(w+x).
$$
It suffices to show continuity of the unique solution in the above problem, as  
a function of $x$.  This can be established using results from variational 
analysis, provided some conditions are met on
the bi-criterion function $f_0(w,x)=f(w+x)$. In particular, Corollary 7.43 in
\citet{rockafellar2009variational} implies that the unique minimizer in the
above problem is continuous in $x$, provided $f_0$ is a closed proper  
convex function that is level-bounded in $w$ locally uniformly in $x$.   By 
assumption, $f$ is a closed proper convex function (it is Legendre), and thus so
is $f_0$. The level-boundedness condition can be checked as follows. Fix any
$\alpha \in \R$ and $x \in \R^n$. The $\alpha$-level set $\{w : f(w+x) \leq
\alpha\}$ is bounded since $x  \mapsto f(x+w)$ has no directions of recession
(to see that this implies boundedness of all level sets, \eg, combine Theorem
27.1 and Corollary 8.7.1 of \citet{rockafellar1970convex}).  Meanwhile, for
any $x' \in \R^n$,  
$$
\{ w : f(w+x') \leq \alpha \} = \{w : f(w+x) \leq \alpha\} + x'-x.
$$
Hence, the $\alpha$-level set of $f_0(\cdot,x')$ is uniformly bounded for 
all $x'$ in a neighborhood of $x$, as desired.  This completes the proof.
\end{proof}

\subsection{Proof of Lemma \ref{lem:invar_glm}}

The proof is similar to that of Lemma 10 in
\citet{tibshirani2012degrees}. Let $\cB,s$ be the boundary set and signs of an  
arbitrary optimal subgradient in \smash{$\hgamma(y)$} in
\eqref{eq:genlasso_g}, and let $\cA,r$ be the active set and active signs
of an arbitrary solution in \smash{$\hbeta(y)$} in \eqref{eq:genlasso_g}.  (Note
that \smash{$\hgamma(y)$} need not correspond to \smash{$\hbeta(y)$}; it may 
be a subgradient corresponding to another solution in \eqref{eq:genlasso_g}.)    

By (two applications of) Lemma \ref{lem:local_glm}, there exist neighborhoods 
$U_1,U_2$ of $y$ such that, over $U_1$, optimal subgradients exist with boundary
set and boundary signs $\cB,s$, and over $U_2$, solutions exist with active set
and active signs $\cA,r$.  For any $y' \in U=U_1 \cap U_2$, by Lemma
\ref{lem:implicit_glm} and the uniqueness of the fit from Lemma
\ref{lem:basic_g}, we have  
$$
X\hbeta(y) = \nabla \psi^* \Big( P^{\psi^*}_{y-K_{\cB,s}}\big(\nabla \psi(0)   
\big)\Big) = \nabla \psi^* \Big( P^{\psi^*}_{y-K_{\cA,r}}\big(\nabla \psi(0)   
\big)\Big),
$$
and as $\nabla \psi^*$ is a homeomorphism, 
\begin{equation}
\label{eq:proj_ab}
 P^{\psi^*}_{y'-K_{\cB,s}}\big(\nabla \psi(0) \big) =
 P^{\psi^*}_{y'-K_{\cA,r}}\big(\nabla \psi(0) \big).
\end{equation}
We claim that this implies \smash{$\nul(P_{\nul(D_{-\cB})} X^T) =
  \nul(P_{\nul(D_{-\cA})} X^T)$}. 

To see this, take any direction \smash{$z \in \nul(P_{\nul(D_{-\cB})} X^T)$},
and let $\epsilon>0$ be sufficiently small so that $y'=y+\epsilon z \in U$.  
From \eqref{eq:proj_ab}, we have   
$$
P^{\psi^*}_{y'-K_{\cA,r}}\big(\nabla \psi(0) \big) = 
P^{\psi^*}_{y'-K_{\cB,s}}\big(\nabla \psi(0) \big) = 
P^{\psi^*}_{y-K_{\cB,s}}\big(\nabla \psi(0) \big) = 
P^{\psi^*}_{y-K_{\cA,r}}\big(\nabla \psi(0) \big),
$$
where the second equality used \smash{$y'-K_{\cB,s}=y-K_{\cB,s}$}, and the third
used the fact that \eqref{eq:proj_ab} indeed holds at $y$.  Now consider the
left-most and right-most expressions above.  For these two 
projections to match, we must have \smash{$z \in \nul(P_{\nul(D_{-\cA})} X^T)$}; 
otherwise, the affine 
subspaces \smash{$y'-K_{\cA,r}$} and \smash{$y-K_{\cA,r}$} would be parallel, 
in which case clearly the projections cannot coincide.  Hence, we have shown
that \smash{$\nul(P_{\nul(D_{-\cB})} X^T) \subseteq \nul(P_{\nul(D_{-\cA})}
  X^T)$}. The reverse inclusion follows similarly, establishing the desired
claim.    

Lastly, as $\cB,\cA$ were arbitrary, the linear subspace
\smash{$L=\nul(P_{\nul(D_{-\cB})} X^T)=\nul(P_{\nul(D_{-\cA})} X^T)$} must be
unchanged for any choice of boundary set $\cB$ and active set $\cA$ at $y$,
completing the proof.   
\hfill\qedsymbol

\bibliographystyle{plainnat}
\bibliography{refs}

\end{document}